\documentclass[reqno]{amsart}
\usepackage[colorlinks=true, pdfstartview=FitV, linkcolor=blue,citecolor=red, urlcolor=blue]{hyperref}
\usepackage{mathtools}
\usepackage{amsmath,amsthm,amssymb,mathrsfs}
\usepackage{amsfonts}
\usepackage{color}
\usepackage{hyperref}
\usepackage{url}
\usepackage{enumitem}   
\newcommand{\bburl}[1]{\textcolor{blue}{\url{#1}}}

\newcommand{\overbar}[1]{\mkern 1.5mu\overline{\mkern-1.5mu#1\mkern-1.5mu}\mkern 1.5mu}

\newtheorem{thm}{Theorem}[section]
\newtheorem{cor}[thm]{Corollary}
\newtheorem{lem}[thm]{Lemma}
\newtheorem{prop}[thm]{Proposition}
\theoremstyle{definition}

\theoremstyle{definition}
\newtheorem{defin}{Definition}
\newtheorem{conjec}[thm]{Conjecture}
\theoremstyle{remark}
\newtheorem{rem}[thm]{Remark}

\newcommand\be{\begin{equation}}
\newcommand\ee{\end{equation}}
\newcommand\ben{\begin{enumerate}}
\newcommand\een{\end{enumerate}}

\newcommand{\QQ}{\mathbb {Q}}
\newcommand{\CC}{\mathbb {C}}
\newcommand{\p}{{\mathfrak p}}

\newcommand{\OO}{{\mathcal O}}

\newcommand{\1}{{\bf 1}}
\newcommand{\edv}{\mathrel\Vert} 

\def\D{\Delta}
\def\z{\bar{z}}

\def\b{\bar{b}}
\def\w{\overbar{w}}
\def\aa{\bar{a}}
\def\cc{\bar{c}}

\def\om{\bar{\omega}}
\def\H{H_{k,\D}}
\def\F{F_{k,D}}
\def\P{\mathcal{P}}
\def\sl2{\textbf{SL}(2,\mathbb{\OO})}
\def\psl2{\textbf{PSL}(2,\mathbb{\OO})}
\def\pgl2{\textbf{PGL}(2,\mathbb{\OO})}

\def\a{\alpha_{k,\D}}
\def\e{\varepsilon}

\numberwithin{equation}{section}

\title[{Sums of Binary Hermitian Forms}]{From Binary Hermitian Forms to parabolic cocycles of Euclidean Bianchi groups}

\author{Cihan Karabulut}
\email{\textcolor{blue}{\href{karabulutc@wpunj.edu}{karabulutc@wpunj.edu}}}
\address{William Paterson University, Wayne, New Jersey 07470, USA}

\thanks{The author was partly supported by Assigned Release Time (ART) program for research from William Paterson University}

\subjclass[2010]{Primary 11E39, 11M06, 11F75; Secondary 11J70}

\keywords{Binary Hermitian forms, special values, parabolic cohomology.}

\date{\today}

\begin{document}

\begin{abstract}
We study a family of functions defined in a very simple way as sums of powers of binary Hermitian forms with coefficients in the ring of integers of an Euclidean imaginary quadratic field $K$ with discriminant $d_K$. Using these functions we construct  a nontrivial cocycle belonging to the space of parabolic cocycles on Euclidean Bianchi groups. We also show that the average value of these functions is related to the special values of $L(\chi_{d_K},s)$. Using the properties of these functions we give new and computationally efficient formulas for computing some special values of $L(\chi_{d_K},s)$.
\end{abstract}

\maketitle

\tableofcontents


\section{Introduction}

In \cite{Za2}, Zagier studies the following family of functions: Let $D>0$ be a non-square integer satisfying $D \equiv 0,1$ (mod 4) and $k>0$ be an even integer. Then, define $\F:\mathbb{R}\rightarrow\mathbb{R}$ as follows, 
\begin{equation}\label{Zfunc1}
F_{k,D}(x):=\sum_{\substack{\text{disc}(Q)=D \\ a<0<Q(x) }} Q(x)^{k-1}
\end{equation}
where $Q(X)=aX^2+bX+c$ with $(a,b,c)\in \mathbb{Z}^3$ and $\text{disc}(Q)=b^2-4ac$. When $D$ is a square discriminant, one has to add a simple correction term consisting of the $k$-th Bernoulli polynomial. But for the sake of brevity and simplicity, let us assume that $D$ is a non-square discriminant.

Zagier shows that these functions have many surprising properties and are intimately related to modular forms of integral weight and half integral weight, certain values of Dirichlet $L$-functions, Diophantine approximation, continued fractions, and Dedekind sums. 

For example, Zagier finds that  $F_{2,5}$ is constant with value $F_{2,5}(x)=F_{2,D}(0)=2$ for any $x\in \mathbb{R}$ despite the fact that there are infinitely many quadratic polynomials contributing to the sum when $x$ is irrational. More generally, he proves the following theorem: 

\begin{thm}[Zagier 1999]\label{ZT}
	Let $D$ be a positive non-square discriminant and $k$ a positive even integer. Then	
	\begin{enumerate}
		\item $F_{2,D}(x)$ has a constant value $\alpha_D$ for all $x \in \mathbb{R}$,
		\item $F_{4,D}(x)$ has a constant value $\beta_D$ for all $x \in \mathbb{R}$,
		\item For $k\geq 6$, $F_{k,D}$ is no longer constant but is a linear combination, with coefficients depending on $D$ and $k$, of a finite collection of functions depending only on $k$.
	\end{enumerate}
\end{thm}

To get an explicit formula for $\alpha_{D}$ and $\beta_{D}$, Zagier uses the above theorem to write $\alpha_{D}$ as $F_{2,D}(0)$ and  $\beta_{D}$ as $F_{4,D}(0)$ which gives the following formulas:
\begin{equation}
\alpha_D=\sum_{\substack{b^2-4ac=D \\a<0<c}}c=\sum_{|b|<\sqrt{D}}\sigma_{1}\left(\frac{D-b^2}{4}\right),
\end{equation}

\begin{equation}
\beta_D=\sum_{\substack{b^2-4ac=D \\a<0<c}}c^3=\sum_{|b|<\sqrt{D}}\sigma_{3}\left(\frac{D-b^2}{4}\right)
\end{equation}
where $(a,b,c)\in \mathbb{Z}^3$. The sums involving the divisor functions $\sigma_{1}$ and $\sigma_{3}$ have appeared in the literature before as the values of Dirichlet $L$-function $L(\chi_D,s)$ associated to the Kronecker symbol $\chi_D:=\left(\frac{D}{\cdot}\right)$ at $s=-1$ and $s=-3$. In fact, Cohen \cite{Co1,Co2} and Zagier \cite{Za1} using the results of Siegel \cite{Si} showed that 
\begin{align}
\alpha_{D}=&-5L(\chi_D,-1),\label{coh-zag1}\\
\beta_{D}=&L(\chi_D,-3)\label{coh-zag2}.
\end{align}

The formulas in (\ref{coh-zag1}) and (\ref{coh-zag2}), which we refer to as Cohen-Zagier type formulas,  are computationally very efficient and give a fast $O(D^{1/2+\epsilon})$ method for computing the special values $s=-1$ and $s=-3$. The usual formulas for computing special values in terms of finite character sums are $O(D)$ and are not very practical for actual computations, except for small conductors (see Remarks on page 194 in \cite{Co3}).

In recent years, there has been a renewed interest in Zagier's paper, with \cite{Be,Ja,JaRa,Wo} all further exploring the functions defined in (\ref{Zfunc1}). In this paper, generalizing in a different direction,  we introduce an analogue of Zagier's functions $\F$ by using binary Hermitian forms and prove a result similar to Theorem \ref{ZT} for these functions. As an application, we get new formulas, similar to the ones given in (\ref{coh-zag1}) and (\ref{coh-zag2}).

Let $K=\mathbb{Q}(\sqrt{-d})$ with $d=1,2,3,7,11$ be an Euclidean imaginary quadratic field with discriminant $d_K$ and $\OO_d$ be the ring of integers of $K$. For a positive integer $\D>0$ which is not a norm of an algebraic integer in $\OO_d$ and an odd integer $k\geq 1$, we define the function  $H_{k,\D}:\mathbb{C} \to \mathbb{R}$ as
\begin{equation}\label{mainfunc}
H_{k,\D}(z):=\sum_{\substack{N(b)-ac=\D \\ a<0<h(z)}} h(z)^{k}\\
\end{equation}
where  $\displaystyle h(z)=a|z|^2+bz+\b\z+c$ with $a,c \in \mathbb{Z}$ and $ b \in \OO_d$ is a binary Hermitian with discriminant $-\D$. Our first result is the following theorem which is analog of Theorem \ref{ZT}:

\begin{thm}\label{mainthmintro}
	Let $\D$ be a positive integer which is not a norm of an algebraic integer in $\OO_d$ and $k\geq1$ be an odd integer. Then	
	\begin{enumerate}
		\item $H_{1,\D}(z)$ has a constant value $\alpha_{1,\D}$ for all $z \in \mathbb{C}$,
		\item $H_{3,\D}(z)$ has a constant value $\alpha_{3,\D}$ for all $z \in \mathbb{C}$ when $d=1,3,7,$
		\item $H_{5,\D}(z)$ has a constant value $\alpha_{5,\D}$ for all $z \in \mathbb{C}$ when $d=3$
	\end{enumerate}
where $\alpha_{k,\D}=\H(0)$.
\end{thm}

The proof of Theorem \ref{mainthmintro} is presented in Section \ref{section2}. The main idea in proving Theorem \ref{mainthmintro} is the observation that any $z\in K$, since $K$ is Euclidean, can be reduced to zero by a finite number of iterations of the transformations $z\mapsto z+\lambda$ and $z\mapsto \frac{1}{z}$. The $\H$ is invariant under the transformation $z\mapsto z+\lambda$ (see Proposition \ref{H-properties}) and transforms up to a certain polynomial $P_{k,\D} $ under the transformation $z\mapsto \frac{1}{z}$ (Propositions \ref{S-act}). These polynomials satisfy various functional equations, and we use them to deduce that they are of a particular form when $k=1,3,5$. This allows us to conclude that for suitable choices of $d$, $\H(z)=\H(0)$ for any $z\in K$ when $k=1,3,5$. The fact that $H_{1,\D}(z)=\alpha_{1,\D}$, $H_{3,\D}(z)=\alpha_{3,\D}$ and $H_{5,\D}(z)=\alpha_{5,\D}$ for all $z\in \mathbb{C}$ is deduced using the continuity of $\H$ which is our next result.

\begin{thm}\label{cont-intro}
	Let $\D$ be a positive integer which is not a norm of an algebraic integer in $\OO_d$ and $k\geq 1$ be an odd integer. Then $\H$ is continuous for all $z \in \mathbb{C}$.
\end{thm}

We discuss the convergence and the continuity of $H_{k,\D}$ and present the proof of Theorem \ref{cont-intro} in Section \ref{cont}. When $k \geq 3$, the proof of Theorem \ref{cont-intro} (see Theorem \ref{contk>4}) is straightforward since  $h(z)=O(1/a_h)$ for all $h$ occurring in Definition (\ref{mainfunc}) and there are only $O(1)$ functions of the form $h(z)=a|z|^2+bz+\b\z+c$ for each value of $a$. So, the sum defining $H_{k,\D}$ converges at most like $\sum_{a>0}a^{-k}$. However, this argument fails when $k=1$. In that case, we deduce the continuity of $H_{k,\D}$ by showing that the binary Hermitian forms appearing in (\ref{mainfunc}) for a fixed $z\in \CC$ decrease to zero exponentially. This is done by describing the binary Hermitian forms appearing in (\ref{mainfunc}) for a fixed $z\in \CC$ using the the nearest integer continued fraction algorithm of $z$ over the Euclidean imaginary quadratic fields. 

Since $H_{k,\D}$ is $\OO_d$-invariant, it has a well defined average value. In Section \ref{averagevalue}, we prove the following theorem which gives the average value of $H_{k,\D}$ in terms of $L(\chi_{d_K},k)$ and $\zeta_{\mathbb{Q}}(k)$:

\begin{thm}\label{avg-intro} Let $\D$ be a positive integer which is not a norm of an algebraic integer in $\OO_d$ and $k\geq 1$ be an odd integer. Then, the average value of $\H$ is given by
	\begin{equation}
	\langle \H(z) \rangle_{av}=\frac{2\pi\D^{k+1}}{(k+1)\sqrt{|d_K|}}\theta(\D,k+1)\zeta_{\mathbb{Q}}(k+1)L(\chi_{d_K},k+2)^{-1}
	\end{equation}	
	where $\theta(\D,k)$ is a finite Euler product whose factors depend on the prime factors of $\D$ and $d_K$.
\end{thm}

The proof (see Theorem \ref{avg}) consists of showing that the integral of $\H$ over the fundamental parallelogram $\P$ for the lattice $\OO_d$ is, up to a constant, equal to the zeta function of binary Hermitian forms with coefficients in $\OO_d$ and discriminant $-\D$. These zeta functions were studied by Elstrodt, Grunewald, and Mennicke \cite{EGM1,EGM2} while developing the theory of representation numbers of binary Hermitian forms over $\OO_d$ which parallels the classical theory of representation of integers by binary quadratic forms over $\mathbb{Z}$.

As an application of  Theorem \ref{mainthmintro} and Theorem \ref{avg-intro} we immediately get the following formulas (see Corollary \ref{Lformulas}):
\begin{enumerate}
	\item Let $d=1,2,3,7,11$ and $d_K$ be the discriminant of $K=\mathbb{Q}(\sqrt{d})$. Then,
	\begin{align}
		L(\chi_{d_K},3)=&\frac{\pi^3\D^2}{6\alpha_{2,\D}\sqrt{|d_K|}}\ \theta(\D,2)\label{lformula1},\\
		L(\chi_{d_K},-2)=&\frac{-|d_K|^2\D^2}{12\alpha_{2,\D}}\ \theta(\D,2)\label{lformula2}.
	\end{align}
	\item Let $d=1,3,7$ and $d_K$ be the discriminant of $K=\mathbb{Q}(\sqrt{d})$. Then,
	\begin{align}
		L(\chi_{d_K},5)=&\frac{\pi^5\D^4}{180\alpha_{4,\D}\sqrt{|d_K|}}\ \theta(\D,4),\\
		L(\chi_{d_K},-4)=&\frac{|d_K|^4\D^4}{120\alpha_{4,\D}}\ \theta(\D,4).
	\end{align}
	
	\item Let $d=3$ and $d_K$ be the discriminant of $K=\mathbb{Q}(\sqrt{d})$. Then,
	\begin{align}
		L(\chi_{d_K},7)=&\frac{\pi^7\D^6}{2835\alpha_{6,\D}\sqrt{3}}\ \theta(\D,6),\\
		L(\chi_{d_K},-6)=&\frac{-81\D^6}{28\alpha_{6,\D}}\ \theta(\D,6).
	\end{align}
\end{enumerate} 
where $\alpha_{1,\D}$,$\alpha_{1,\D}$ and $\alpha_{1,\D}$ are the values of $H_{1,\D}(0)$, $H_{1,\D}(0)$ and $H_{1,\D}(0)$   given explicitly by
				
\begin{align}
\alpha_{1,\D}:&=\sum_{\substack{N(b)-ac=\D \\a<0<c}}c 
=\sum_{\substack{N(b)<\D }}\sigma_1\left(\D-N(b)\right),\\
\alpha_{3,\D}:&=\sum_{\substack{N(b)-ac=\D \\a<0<c}}c^3 
=\sum_{\substack{N(b)<\D }}\sigma_{3}\left(\D-N(b)\right),\\
\alpha_{5,\D}:&=\sum_{\substack{N(b)-ac=\D \\a<0<c}}c^5 
=\sum_{\substack{N(b)<\D }}\sigma_{5}\left(\D-N(b)\right).
\end{align}

The above formulas for special values can be seen as Cohen-Zagier type formulas and as such they are computationally very efficient. In fact, using our formulas to compute the special values at $s=-6,-4,-2,3,5,7$ requires only the computation of $\theta(\D,k)$ and  $\alpha_{k,\D}$ for $k=1,3,5$. Furthermore, since the left hand side of our formulas is independent of $\D$, we can choose the most optimal $\D$, i.e. usually the smallest $\D$ that is not a norm, when computing $\theta(\D,k)$ and  $\alpha_{k,\D}$.

As a simple example, we use our formulas to compute the aforementioned special values when $K=\mathbb{Q}(i)$ with $\OO_1=\mathbb{Z}[i]$, $d_K=-4$ and $\chi_{d_K}=\left(\frac{-4}{\cdot}\right)$. Letting $\D=3$, a simple computation shows that $\theta(3,1)=\frac{5}{6}$, $\theta(3,3)=\frac{425}{432}$, $\alpha_{1,3}=20$ and $\alpha_{3,3}=68$. Finally, plugging all these into appropriate above formulas gives

\begin{align*}
 L(\chi_{-4},3)=&\frac{3\pi}{16}\ \zeta_{\mathbb{Q}}(2)=\frac{\pi^3}{32},\\
L(\chi_{-4},5)=&\frac{75\pi}{256}\ \zeta_{\mathbb{Q}}(4)=\frac{5\pi^5}{1536},\\
L(\chi_{-4},-2)=&\frac{-|-4|^2\cdot 3^2}{12\cdot 20}\cdot\frac{5}{6}=-\frac{1}{2},\\
L(\chi_{-4},-4)=&\frac{32\cdot 3^4}{15\cdot 68}\cdot\frac{425}{432}=\frac{5}{2}.
\end{align*}

In the last section, we discuss the construction of a nontrivial parabolic cocycle on the Euclidean Bianchi group $\textbf{PSL}(2,\mathbb{\OO}_d)$ using $\H$. We carry this out for each $d\in \{1,2,3,7,11\}$ by showing that the space of parabolic cocycles on $\textbf{PSL}(2,\mathbb{\OO}_d)$ is isomorphic to a certain space of polynomials $W_{k,k}$ which is defined using the presentation of $\textbf{PSL}(2,\mathbb{\OO}_d)$ and the cocycle condition. We then show that $P_{k,\D}$, which is the polynomial that $\H$ transforms by under the map $z\mapsto \frac{1}{z}$, belongs to $W_{k,k}$. Consequently, $P_{k,\D}$, which is defined using $\H$, corresponds to a parabolic cocycle on $\textbf{PSL}(2,\mathbb{\OO}_d)$. In view of the fact that $P_{k,\D}$ is invariant under the transformation $z\mapsto uz$ where $u\in \OO_d$ is a unit, we study decomposition of $W_{k,k}$ as a direct sum of the eigenspaces corresponding to the eigenvalues of the linear operator associated to the linear map $z\mapsto uz$. This allows us to precisely describe the subspace of $W_{k,k}$ containing $P_{k,\D}$. We also compute the dimensions of $W_{k,k}$ and its subspace containing $P_{k,\D}$ for some odd values of $k$ using a computer program that we wrote in Sagemath \cite{Sa}. In each case, based on the numerical evidence,  we conjecture dimension formulas for these spaces.

Finally, it is fair to say the cocycle property of $\H$ is perhaps the most interesting and also the deepest property of $\H$ as it connects $\H$ to cohomology of Euclidean Bianchi groups which are central to the study of Bianchi modular forms. This connection to cohomology is a promising one and is worth studying more. Accordingly, we end the section with saying a few more words on this, and also describing some collaborative work \cite{FKW} and other related questions. 

\emph{Acknowledgements.} This work is continuation of my PhD thesis. I would like to express my gratitude to my advisor, Gautam Chinta, for his support and guidance throughout the completion of my thesis. I would also like to thank Jorge Fl\'orez, Tian An Wong and Don Zagier for helpful conversations. Finally, I would like to thank the anonymous referee for a careful reading and numerous suggestions which greatly improved the manuscript.
\section{Definition and Elementary Properties of $\H$}\label{section2}
\subsection{Preliminaries}
In this section we establish some notation and conventions which are used throughout the rest of this paper.  We fix $d=1,2,3,7,11$ and let $K=\QQ(\sqrt{-d})$ be the corresponding Euclidean imaginary quadratic number field with discriminant $d_K$. We let $\OO_d=\OO_K$ be the ring of integers of $K$ and choose $\{1,\omega\}$ with
\begin{gather*}
\omega:=\frac{d_K+\sqrt{d_K}}{2}
\end{gather*}
 as a fixed $\mathbb{Z}$-basis of $\OO_d$.

Let $V_n(\mathbb{C})$ be the space of polynomials in $z$ of degree less than or equal to $n\in \mathbb{Z}_{\geq 0}$ with coefficients in $\mathbb{C}$. For a polynomial $P(z)$ in $V_n(\CC)$ and a matrix $\gamma=\begin{psmallmatrix}
a&b\\ c&e 
\end{psmallmatrix} \in \textbf{SL}(2,\mathbb{C})$ we have the right action 

\begin{equation}\label{onevariable}
(P|\gamma):=(cz+e)^{n}P\left(\frac{az+b}{cz+e}\right).
\end{equation}
Because $\begin{psmallmatrix}
-1&0\\ 0&-1 
\end{psmallmatrix}$ acts trivially, $V_n(\mathbb{C})$ is also $\textbf{PSL}(2,\mathbb{C})$-module. Now, consider the $\textbf{PSL}(2,\mathbb{C})$-module $V_{n,n}(\CC):=V_n(\CC)\otimes_\CC\overbar{V_n(\CC)}$ where the overline indicates that the action of $\textbf{SL}(2,\mathbb{C})$ on the second factor is twisted with complex conjugation; that is for $P(z,\z)$ in $V_{n,n}(\mathbb{C})$ and a matrix $\gamma=\begin{psmallmatrix}
a&b\\ c&e 
\end{psmallmatrix} \in \textbf{PSL}(2,\mathbb{C})$ we have the right action

\begin{equation}\label{twovariable}
(P|\gamma):=(cz+e)^{n}\overbar{(cz+e)}^nP\left(\frac{az+b}{cz+e},\frac{\aa\z+\b}{\cc\z+\bar{e}}\right).
\end{equation}

We also extend the action of $\textbf{PSL}(2,\mathbb{C})$ on $V_{n,n}(\mathbb{C})$ by linearity to an action of the group ring $\mathbb{Z}[\textbf{PSL}(2,\mathbb{C})]$. The groups $\textbf{PSL}(2,\mathbb{\OO}_d)$ play an important role in our work, so we let $\Gamma_d=\textbf{PSL}(2,\mathbb{\OO}_d)$. Finally, to simplify the notation the letter $\CC$ sometimes will be omitted when referring to different $\textbf{PSL}(2, \CC)$-modules throughout the paper.

\subsection{Sums of Binary Hermitian Forms with coefficients in $\OO_d$}
Before we give the definition of $\H$ we recall some relevant facts from the theory of binary Hermitian forms as discussed in \cite{EGM1,EGM2}. Let $A$ be a $2 \times 2$ matrix  with entries in $\mathbb{C}$. $A$ is said to be a Hermitian matrix if 
\begin{gather*}
A=\bar{A}^t
\end{gather*}
where $\bar{A}^{t}$ is obtained from $A$ by applying complex conjugation to each of the entries and then taking the transpose or vice versa. Let $R$ be a subring of $\mathbb{C}$ which is closed under conjugation, i.e. $R=\overbar{R}$. We denote by $\mathcal{H}(R)$ the set of all $2 \times 2$ Hermitian matrices with entries in $R$. Trivially, $h\in\mathcal{H}(R)$ if and only if $h=\begin{psmallmatrix}
a&b\\ \b&c
\end{psmallmatrix}$ with $a,c \in R\cap \mathbb{R}$ and $b \in R$. Every $h\in\mathcal{H}(R)$ defines a binary Hermitian form with coefficients in $R$. If $h=\begin{psmallmatrix}
a&b\\ \b&c
\end{psmallmatrix}$ then the associated binary Hermitian form is the map  $h:\mathbb{C\times C} \rightarrow \mathbb{R}$ defined by, $$h(z,w)=\left(\begin{matrix}
z&w\\
\end{matrix} \right)\left(\begin{matrix}
a&b\\ \b&c
\end{matrix} \right) \left(
\begin{matrix}
\z\\
\w\\
\end{matrix}
\right)=a|z|^{2}+bz\w+\b\overbar{z}w+c|w|^{2}.$$ 

We shall often call an element $h \in \mathcal{H}(R)$ a binary Hermitian form with coefficients in $R$, and we shall use the notation $a_h,b_h, \b_h,c_h$ to refer to the individual matrix entries of $h$ when the need arises. The discriminant $\Delta(h)$ of $h \in \mathcal{H}(R)$ is defined as 
\begin{gather*}
\Delta(h)=\text{det}(h),
\end{gather*}
and $\mathcal{H}(R,\D)$ denotes the set of all $h \in \mathcal{H}(R)$ with $\Delta(h)=\D$.

Similar to the action of the $\textbf{GL}(2,\mathbb{Z})$ on binary quadratic forms, we have the action of $\textbf{GL}(2,R)$ on $\mathcal{H}(R)$ given by the formula
\begin{equation}\label{Psl2act}
\sigma (h)=  \overbar{\sigma}^t h \sigma
\end{equation}
where $\sigma \in \textbf{GL}(2,R)$ and $h \in \mathcal{H}(R)$.

\begin{rem}
	The action defined in (\ref{Psl2act}) differs slightly from that of \cite{EGM1,EGM2}, but it is more convenient for our purposes. 
\end{rem}

For our purposes we let $R=\OO_d$ and define $\H$ as follows.
\begin{defin}\label{myfuncdef}
Let $k\geq 1$ be an odd integer and $\D$ be a positive integer that is not a norm of an algebraic integer in $\OO_d$. Then, we define the function  $H_{k,\D}:\mathbb{C} \rightarrow \mathbb{R}$ as
\begin{align}\label{myfunc}
H_{k,\D}(z):&=\sum_{\substack{h\in \mathcal{H}(\mathbb{\OO}_d,-\D)\\ a_h<0 }}\max\left(0, h(z,1)^{k}\right)\\
&=\sum_{\substack{N(b)-ac=\D\\ a<0}} \max \left(0,\left(a|z|^2+bz+\b \z+c\right)^{k}\right)
\end{align}
where $a,c \,\in\, \mathbb{Z},\, b\,\in\, \OO_d$, and $N(b)=|b|^2=b\overbar{b}$ is the usual norm on $K$.
\end{defin}

We assume that $k\geq 1$ is an odd integer and $\D$ is a positive integer that is not a norm of an algebraic integer in $\OO_d$ for the rest of the paper unless otherwise is stated.

The convergence and the continuity of $\H$ is discussed in section \ref{cont}. In particular, we show that $\H$ converges  and is a continuous function for any $k\geq 1$, $\D>0$ as in the definition of $\H$. Let us assume for now that  $\H$ converges and is a continuous function for all $z \in \mathbb{C}$.

Note that the set of all $h(z,1)$ appearing in equation (\ref{myfunc}) is a subset of $V_{1,1}(\OO_d)$ and an elementary computation shows that the action of $\gamma =\begin{psmallmatrix}
	p&q\\ r&s
\end{psmallmatrix}\in \Gamma_d$ on $h(z,1)$ as an element of $V_{1,1}(\OO_d)$ corresponds to the action of $\bar{\gamma}\in \Gamma_d$ on $h$ as an element of $\mathcal{H}(\mathcal{O}_d,\D)$. That is, if $P(z,\z)=h(z,1)$ then  

\begin{gather*}\label{act-on-h}
P|\gamma=\left(\begin{matrix}
z&1\\
\end{matrix}  \right)\overbar{\gamma}(h)\left(\begin{matrix}
\z\\1
\end{matrix} \right).
\end{gather*}
Using this observation we can extend the action of $\textbf{PSL}(2,\mathbb{C})$ to $\H$ by defining   

\begin{align*}
	\H|\gamma&:=\sum_{\substack{h\in \mathcal{H}(\mathbb{\OO}_d,-\D) \\ a_h<0}}\max\left(0, \left(\left(\begin{matrix}
	z&1\\
	\end{matrix}  \right)\overbar{\gamma}(h)\left(\begin{matrix}
	\z\\1
	\end{matrix} \right)\right)^{k}\right)\\
	&=|rz+s|^{2k}\H\left(\frac{pz+q}{rz+s}\right)	
\end{align*}

Let us now discuss some properties of $\H$ and its behavior under the action of certain elements of $\Gamma_d$. 

\begin{prop}\label{H-properties} \leavevmode	
\begin{itemize}
	\item [(1)] $H_{k,\D}(\z)=H_{k,\D}(z)$ for any $z\in \CC$.\\
	\item [(2)] $H_{k,\D}(u z)=H_{k,\D}(z)$ for any unit $u \in \OO_d$.\\
	\item [(3)] $H_{k,\D}(z+\lambda)=H_{k,\D}(z)$
	for any $\lambda \in \OO_d$.
\end{itemize}
\end{prop}

\begin{proof}
We have 
\begin{align*}
\H(\z)&=\sum_{\substack{N(b)-ac=\D\\ a<0}} \max \left(0,\left(a|\z|^{2}+b\z+\b z+c\right)^{k}\right)\\
&=\sum_{\substack{N(b)-ac=\D\\ a<0}} \max \left(0,\left(a|z|^{2}+bz+\b\z+c\right)^{k}\right)
\end{align*}
since if $b \in \OO_d$ is a solution to $N(b)-ac=\D$ then so is $\b \in \OO_d$. The proof of $\H(uz)=\H(z)$ is similar since if $u\in \OO_d$ is a unit then $$a|uz|^{2}+buz+\b\overbar{uz}+c=a|z|^{2}+buz+\b\overbar{u}\z+c,$$
and $b \in \OO_d$ can be replaced by $bu \in \OO_d$ or $\b\overbar{u}\in \OO_d$ in the equation $N(b)-ac=\D$.\\
	
\noindent As for the proof of the last statement, notice that $\H(z+\lambda)=\H|A$ where $A=\begin{psmallmatrix}
1&\lambda\\0&1
	\end{psmallmatrix}\in \Gamma_d.$  Since $$\overbar{A}(h)=A^th\overbar{A}=\begin{pmatrix}
	a&a\overbar{\lambda}+b\\a\lambda+\b&a\lambda\overbar{\lambda}+b\lambda+\b\overbar{\lambda}+c
	\end{pmatrix}$$
is just another binary Hermitian form with discriminant $\D$ and negative first entry, we see that $\H$ is invariant under the action of $A$. 
\end{proof}
As mentioned in the introduction, we use the Euclidean algorithm, which consists of finite iterations of the maps $z \mapsto z+\lambda $ and $z\mapsto \frac{1}{z}$, to prove that $\H$ is constant in some special cases. We already saw that $\H$ is invariant under the map $z \mapsto z+\lambda $, but we will see next that $\H$ is not invariant under map $z\mapsto \frac{1}{z}$ and transforms by a polynomial which is defined as:  
\begin{defin}\label{PPolyforZ[i]}  We define the polynomials $P_{k,\D}$ in $z$ and $\z$ as
\begin{align}
P_{k,\D}(z,\z):&=	\sum_{\substack{N(b)-ac=\D \\ c<0<a}} \left(az\z+bz+\b\z+c\right)^{k}\label{P_k-def2}
\end{align}
where $b\in \OO_d$ and $a,c \in \mathbb{Z}.$ Notice that $P_{k,\D}\in V_{k,k}(\OO_d)$.	
\end{defin}
  
\begin{prop}\label{S-act}
	Let $S=\begin{psmallmatrix}
	0&-1\\1&0
	\end{psmallmatrix}$ and $\1= \begin{psmallmatrix}	1&0\\0&1	\end{psmallmatrix}$. Then, 
	\begin{gather*}
	H_{k,\D}|(S-\1)=P_{k,\D}
	\end{gather*} where $|(S-\1)$ indicates the action of the element $S-\1$ in the group ring $\mathbb{Z}[\textbf{PSL}(2,\CC)]$ on $\H$.
\end{prop}

\begin{proof}
We first prove that $H_{k,\D}|(S-\1)=P_{k,\D}$ for $z\in K$ and use the continuity of $\H$ to conclude that it holds for all $z\in \CC$. Notice that when $z\in K$, the number of terms appearing in the sum for $H_{k,\D}(z)$ is finite. To see this, suppose $\displaystyle h(z,1)=a|z|^{2}+bz+\b\z+c \in \mathcal{H}(\mathbb{\OO}_d,-\D)$ occurs in the sum for $H_{k,\D}(z)$ for a fixed $\displaystyle z=\frac{p}{q}+\frac{r}{s}\sqrt{-d}$ with $p,q,r,s \in \mathbb{Z}$. Then using the identity $-\D(h)=N(\b_h+a_hz)-a_hh(z,1)$ one gets,
	\begin{gather*}
		\D q^{2}s^{2}=N(qs\b+aw)+|a|\left|a|w|^{2}+qsbw+qs\b\overbar{w}+cq^2s^2\right| 
	\end{gather*} 
	where $w=ps+qr\sqrt{-d}$. Since $\left|a|w|^{2}+qsbw+qs\b\overbar{w}+cq^2s^2\right| \geq 1$, the last equation implies that $|a| \leq \D q^{2}s^{2}$ which shows that $a$ is bounded, and this in turn bounds $c$ and  $N(b)$ as well. Thus only finitely many $h(z,1)$ appear in the sum for $H_{k,\D}(z)$ when $z \in K$. On the other hand,  the action of $S-\1$ on $\H$ is given by $$H_{k,\D}|(S-\1)=H_{k,\D}|S-\H,$$ which can be computed 
	 directly as
	\begin{align*}
		&=|z|^{2k} H_{k,\D}\left(\frac{-1}{z}\right)- H_{k,\D}(z)\\
		&=|z|^{2k} H_{k,\D}\left(\frac{1}{z}\right)- H_{k,\D}(z) \quad \text{(since $\H(-z)=\H(z))$}\\
		&=|z|^{2k} \sum_{\substack{N(b)-ac=\D \\ a<0}} \max \left(0,\left(\frac{a}{|z|^{2}}+\frac{b}{z}+\frac{\b}{\z}+c\right)^{k}\right)-\sum_{\substack{N(b)-ac=\D \\ a<0}} \max \left(0,\left(a|z|^{2}+bz+\b\z+c\right)^{k}\right)\\
		&=\sum_{\substack{N(b)-ac=\D \\ a<0}} \max \left(0,\left(a+b\z+\b z+c|z|^{2} \right)^{k}\right)-\sum_{\substack{N(b)-ac=\D \\ a<0}} \max \left(0,\left(a|z|^{2}+bz+\b\z+c\right)^{k}\right)\\
		&=\sum_{\substack{N(b)-ac=\D \\ c<0}} \max \left(0,\left(a|z|^{2}+bz+\b \z+c\right)^{k}\right)-\sum_{\substack{N(b)-ac=\D \\ a<0}} \max \left(0,\left(a|z|^{2}+bz+\b\z+c\right)^{k}\right).
	\end{align*}
	Notice that the summands with $a$ and $c$ both negative in the two previous sums cancel. Moreover $a \neq 0$ and $c \neq 0$ because $\D$ is not a norm in $\OO_d$. Therefore, we have
	\begin{align*}
		&=\sum_{\substack{N(b)-ac=\D \\ c<0<a}} \max \left(0,\left(a|z|^{2}+bz+\b \z+c\right)^{k}\right)-\sum_{\substack{N(b)-ac=\D \\ a<0<c}} \max \left(0,\left(a|z|^{2}+bz+\b\z+c\right)^{k}\right).
	\end{align*}
	Now, applying the fact that $\max(0,X)=-\min(0,-X)$ for any $X\in \mathbb{R}$ gives 
	\begin{align*}
		&=\sum_{\substack{N(b)-ac=\D \\ c<0<a}} \max \left(0,\left(a|z|^{2}+bz+\b \z+c\right)^{k}\right)+\sum_{\substack{N(b)-ac=\D \\ a<0<c}} \min \left(0,\left(-a|z|^{2}-bz-\b\z-c\right)^{k}\right)\\
		&=\sum_{\substack{N(b)-ac=\D \\ c<0<a}} \max \left(0,\left(a|z|^{2}+bz+\b \z+c\right)^{k}\right)+\sum_{\substack{N(b)-ac=\D \\ c<0<a}} \min \left(0,\left(a|z|^{2}+bz+\b\z+c\right)^{k}\right). 
	\end{align*}
Finally, because $\max(0,X)+\min(0,X)=X$ for any $X\in \mathbb{R}$ we get
\begin{align*}
		H_{k,\D}|(S-\1)&=\sum_{\substack{N(b)-ac=\D \\ c<0<a}} \left(a|z|^{2}+bz+\b\z+c\right)^{k}.
\end{align*}
\end{proof}

Let us now define the following numbers which are the values of $\H$ at $z=0$:
\begin{align}
\a:&=\sum_{\substack{N(b)-ac=\D \\a<0<c}}c^k \label{H(0)}\\
&=\sum_{\substack{N(b)<\D }}\sigma_{k}\left(\D-N(b)\right)\label{omegas}.
\end{align}

\begin{cor}
\begin{equation*}
H_{1,\D}|(S-\1)=P_{1,\D}(z,\z)=\alpha_{1,\D}(z\z-1).
\end{equation*}
 \end{cor}
\begin{proof}
	When $k=1$, by Proposition \ref{S-act}, we get 
\begin{equation*}
H_{1,\D}|(S-\1)=\sum_{\substack{N(b)-ac=\D \\ c<0<a}} az\z+bz+\b\z+c.
\end{equation*}	
The coefficient of $z$ and $\z$ is zero because if $b$ and $\b$ are solutions to $N(b)-ac=\D$ then so are $-b$ and $-\b$. Finally, it is easy to see the coefficient of $z\z$ term is $\alpha_{1,\D}$ and the coefficient of the constant term is  $-\alpha_{1,\D}$ since $c<0$. In other words, we have 
$$H_{1,\D}|(S-\1)=\alpha_{1,\D}(z\z-1).$$
\end{proof}

We now have everything we need to show that $H_{1,\D}$ is constant.

\begin{thm}\label{H2d=w2d}
	Suppose that  $H_{1,\D}$ is a continuous function. Then,  $H_{1,\D}(z)=\alpha_{1,\D}$ for all $z \in \mathbb{C}$.
\end{thm}
\begin{proof}
	Let $\Theta_{1,\D}(z):=H_{1,\D}(z) -\alpha_{1,\D}$. Then, $$\Theta_{1,\D}(z+\lambda)=\Theta_{1,\D}(z)\quad \text{and} \quad |z|^{2}\Theta_{1,\D}\left(\frac{1}{z}\right)=\Theta_{1,\D}(z),$$ where $\lambda \in \OO_d $ and $z \in K$. 
	
Since $\OO_d$ has an Euclidean algorithm, any $z \in K$ can be reduced to zero by a finite number of iterations of the transformations $z\mapsto z+\lambda$ and $\displaystyle z\mapsto \frac{1}{z}$. This shows that $$\Theta_{1,\D}(z) =\Theta_{1,\D}(0)=0$$ for all $z \in K$ which implies that $H_{1,\D}(z) =\alpha_{1,\D}$ for all $z \in K$. Assuming the continuity, we conclude that $H_{1,\D}(z) =\alpha_{1,\D}$ for all $z\in \mathbb{C}$.
\end{proof}


The key ingredient that makes the proof of Theorem \ref{H2d=w2d} work is the fact that $P_{1,\D}(z,\z)=\alpha_{1,\D}(z\z-1)$. If we want to use the same argument to show that $\H$ is also constant when $k=3,5$ for the relevant values of $d$ then we need to establish that $P_{3,\D}$ and $P_{5,\D}$ have the same form as $P_{1,\D}$. In the following proposition we list some identities that the polynomials $P_{k,\D}$ satisfy. We use some of these identities in Lemma \ref{S-actk=4} below to show that $P_{3,\D}$ and $P_{5,\D}$ have the desired form. Some of these identities, especially the first two, are quite clear and can be surmised easily given the definition of $P_{k,\D}$. However, the last two  may seem strange, and may look as if they are being pulled out of a hat. But, that is not the case, and they come from combining some of the relations that define the vector space of polynomials containing $P_{k,\D}$. We have chosen them, granted in a seemingly ad-hoc manner, purely for the purpose of deducing that $P_{3,\D}$ and $P_{5,\D}$ have the desired form. We discuss these vector spaces for each $d$ in the last section, and show that they can be identified with the space of parabolic 1-cocycles for $\Gamma_d$ which is related to Bianchi cusp forms on $\Gamma_d$ via generalized Eichler-Shimura isomorphism. We also compute their dimensions for small degrees using a computer program written by the author in SageMath \cite{Sa}, and in each case also conjecture a formula for the dimension of these spaces. Indeed, the computer calculations show that the vector space containing $P_{k,\D}$ is one dimensional for all $d$ when $k=1$, for $d=1,3,7$ when $k=3$, and for $d=3$ when $k=5$.     
\begin{prop}\label{PPolyforZ[i]prop} Let $\e=\left(\begin{smallmatrix}
	-1&0\\ 0&1 
	\end{smallmatrix}\right),\, T=\left(\begin{smallmatrix}
	1&1\\ 0&1 
	\end{smallmatrix}\right),\, T_{\omega}=\left(\begin{smallmatrix}
	1&\omega\\ 0&1 
	\end{smallmatrix}\right)$ where $\omega=\frac{1+\sqrt{-7}}{2}$. Then,
\begin{itemize}
	\item [(1)] $P_{k,\D}(\z,z)=P_{k,\D}(z,\z)$.\\
	\item [(2)] $P_{k,\D}(uz,\bar{u}\z)=P_{k,\D}(z,\z)$ for any unit $u \in \OO_d$.\\
	\item [(3)] $P_{k,\D}|(\1+S)=0$.\\
	\item [(4)] $P_{k,\D}|(\1+TS\e-T)=0$.\\
	\item [(5)] Additionally, when $P_{k,\D}\in V_{k,k}(\OO_7)$  we have
	\[ P_{k,\D}|(\1-T_{\omega}-ST^{-1}T_{\omega}S-TT_{\omega}^{-1}ST_{\omega})=0.\]
\end{itemize} 
\end{prop}

\begin{proof}
Part (1) and (2) are clear from the equation (\ref{P_k-def2}) once we notice that if $b\in \OO_d $ is a solution to $N(b)-ac=\D$ then so are $\b \in \OO_d$ and $bu \in \OO_d$.

For part (3) we have 
\begin{gather*}
P_{k,\D}|(S+\1)=H_{k,\D}|(S-\1)(S+\1)=H_{k,\D}|(S^2-S+S+\1)=H_{k,\D}|(S^2-\1)=0
\end{gather*}	
since $S^2=\1$.
	
As for part (4) we make use of the identities  
\begin{gather*}
\e S=S\e,\ \ T\e=\e T^{-1},\ \ TSTS=ST^{-1}, \ \ \H|\e=\H|T=\H|T^{-1}=\H,
\end{gather*}
and calculate	
\begin{align*}
P_{k,\D}|(\1+TS\e-T)=&H_{k,\D}|(S-\1)(\1+TS\e-T)\\
=&\H|(S-\1+STS\e-TS\e-ST+T)\\
=&\H|(S+TSTS\e-\e T^{-1}S-ST)\\
=&\H(ST^{-1}\e-ST)\quad \\
=&0.
\end{align*}
	
Finally for the proof of part (5) calculating $(S-\1)|(\1-T_{\omega}-ST^{-1}T_{\omega}S-TT_{\omega}^{-1}ST_{\omega})$ gives
\begin{gather*}
S-\1 -ST_{\omega}+T_{\omega}-T^{-1}T_{\omega}S+ST^{-1}T_{\omega}S-STT_{\omega}^{-1}ST_{\omega}+TT_{\omega}^{-1}ST_{\omega},
\end{gather*}
and the result follows after applying the following identities
\begin{gather*}
\H|T_{\omega}=\H|T_{\omega}^{-1}=\H|T^{-1}T_{\omega}=\H, \quad  STT_{\omega}^{-1}ST_{\omega}=T_{\omega}^{-1}ST^{-1}T_{\omega}S.
\end{gather*} 
\end{proof}
\begin{lem}\label{S-actk=4} \leavevmode		
\begin{enumerate}
	\item $P_{3,\D}(z,\z)=\alpha_{3,\D}(z^3\z^3-1)$ when $d=1,3,7$.\\
	\item $P_{5,\D}(z,\z)=\alpha_{5,\D}(z^5\z^5-1)$ when $d=3$.\\
\end{enumerate}	
\end{lem}
\begin{proof}
	
By Definition \ref{PPolyforZ[i]} we have \[P_{3,\D}(z,\z)=\sum_{\substack{N(b)-ac=\D \\ c<0<a}} \left(a|z|^2+bz+\b\z+c\right)^{3} \] where $b\in \OO_d$ and $a,c \in \mathbb{Z}.$	
Using the multinomial theorem we can write $P_{3,\D}(z)$ in the form 
\begin{align*}
P_{3,\D}(z,\z)&=\sum_{\substack{N(b)-ac=\D \\c<0<a}}\left(\sum_{k_1+k_2+k_3+k_4=3}\binom{3}{k_1,k_2,k_3,k_4}a^{k_1}b^{k_2}\b^{k_3}c^{k_4}|z|^{2k_1}z^{k_2}\z^{k_3}\right)\\
&=\sum_{k_1+k_2+k_3+k_4=3}\binom{3}{k_1,k_2,k_3,k_4}\left(\sum_{\substack{N(b)-ac=\D \\c<0<a}}a^{k_1}b^{k_2}\b^{k_3}c^{k_4}\right)|z|^{2k_1}z^{k_2}\z^{k_3}\label{expandedform}.				
\end{align*}
We now use Proposition \ref{PPolyforZ[i]prop} to deduce restrictions on the coefficients of $P_{3,\D}$. Notice that $P_{3,\D}$ is not identically zero. The units in $\OO_1$  are $-1$ and $i$. So, by Proposition \ref{PPolyforZ[i]prop} we have 
\begin{gather*}
P_{3,\D}(-z,-\z)=P_{3,\D}(z,\z), \quad P_{3,\D}(iz,\bar{i}\z)=P_{3,\D}(z,\z).
\end{gather*}	
$P_{3,\D}(-z,-\z)=P_{3,\D}(z,\z)$ implies that $2|(k_{2}+k_{3})$ which, in turn, implies $k_{2}+k_{3}=0,2$, because $k_{2}+k_{3}\leq3$. 	
On the other hand, $P_{3,\D}(iz,\bar{i}\z)=P_{3,\D}(z,\z)$  implies that either $k_{3}$ is odd and $k_{2}+k_{3} \equiv 2$ (mod 4) or $k_3$ is even and $k_{2}+k_{3}\equiv 0$ (mod 4). Since $k_{2}+k_{3}=0,2$, in the first case  we get that $k_{2}=k_{3}=1$ and in the second case we get that $k_{2}=k_{3}=0.$ 
	
When $d=3$, $-1, \zeta=e^{\frac{2\pi i}{3}} \in \OO_3$ are units, and again by part (2) of Proposition \ref{PPolyforZ[i]prop} we have that 
\begin{gather*}
P_{3,\D}(-z,-\z)=P_{3,\D}(z,\z), \quad P_{3,\D}(\zeta z, \bar{\zeta} \z)=P_{3,\D}(z,\z).
\end{gather*}
Again, $P_{3,\D}(-z,-\z)=P_{3,\D}(z,\z)$ implies that $k_{2}+k_{3}=0,2$. On the other hand, $P_{3,\D}(\zeta z, \zeta \z)=P_{3,\D}(z,\z)$ implies that $k_2+2k_3\equiv 0$ (mod 3). Since $k_{2}+k_{3}=0,2$,  we get that $k_{2}=k_{3}=1$ or $k_{2}=k_{3}=0.$ Therefore, in both cases we have 
\begin{align*}
P_{3,\D}(z,\z)&=\sum_{\substack{N(b)-ac=\D \\c<0<a}}\left(\sum_{k_1+k_4=3}\binom{3}{k_1,0,0,k_4}a^{k_1}c^{k_4}|z|^{2k_1}\right)+\\
&\sum_{\substack{N(b)-ac=\D \\c<0<a}}\left(\sum_{k_1+k_4=1}\binom{3}{k_1,1,1,k_4}a^{k_1}|b|^2c^{k_4}|z|^{2k_1+2}\right)\\
&=\sum_{\substack{N(b)-ac=\D \\c<0<a}}\left(a^{3}|z|^6+\left(6a|b|^2+3a^2c\right)|z|^4+\left(6c|b|^2+3ac^2\right)|z|^2+c^3\right)\\
&=\alpha_{3,\D}|z|^{6}+\alpha_{3,\D}'|z|^{4}+\alpha_{3,\D}''|z|^{2}-\alpha_{3,\D} 
\end{align*}
for some $\alpha_{3,\D}', \alpha_{3,\D}'' \in \mathbb{Z}$ depending on $\D$. 

Because $P_{3,\D}|(S+\1)=0$, we see that $\alpha_{3,\D}'= -\alpha_{3,\D}''$ which gives $$P_{3,\D}(z,\z)=\alpha_{3,\D}|z|^{6}+\alpha_{3,\D}'|z|^{4}-\alpha_{3,\D}'|z|^{2}-\alpha_{3,\D}.$$ By part (4) of Proposition \ref{PPolyforZ[i]prop} we have
\begin{equation}\label{P|T-U}
P_{3,\D}(z+1,\z+1)-P_{3,\D}(z,\z)=|z|^{6}P_{3,\D}\left(\frac{z+1}{z},\frac{\z+1}{\z}\right).
\end{equation} 
A direct calculation shows that  $|z|^6-1$ satisfies the equation  (\ref{P|T-U}) but that $|z|^4-|z|^2$ does not. Thus, the middle terms must vanish (i.e. $\alpha_{3,\D}'=0$), and we get $$P_{3,\D}(z)=\alpha_{3,\D}|z|^6-\alpha_{3,\D}=\alpha_{3,\D}(|z|^6-1)$$ which proves the first assertion for $d=1,3$.
	
It remains to prove the first assertion when $d=7$. Since $P_{3,\D}(-z,-\z)=P_{3,\D}(z,\z)$, we know that $k_2+k_3=0,2$	which implies that $(k_2,k_3)=(0,0),(0,2),(2,0),(1,1)$. Using these values for $k_2$ and $k_3$ in the equation (\ref{expandedform}) gives 
\begin{equation*}
P_{3,\D}(z,\z)=\alpha_{3,\D}|z|^6+\delta_1|z|^4+\delta_2|z|^2+\delta_3|z|^2\z^2+\overbar{\delta_3}|z|^2z^2+\delta_4z^2+\overbar{\delta_4}\z^2+\alpha_{3,\D}
\end{equation*}
where $\delta_1,\delta_2\in \mathbb{Z}$ and $\delta_3, \delta_4 \in \OO_{7}.$ Because $P_{3,\D}|(S+\1)=0$, we see that $\delta_1=-\delta_2, \  \delta_3=-\delta_4 $, and thus 
\begin{equation*}
P_{3,\D}(z,\z)=\alpha_{3,\D}(|z|^6-1)+\delta_1(|z|^4-|z|^2)+\delta_3(|z|^2\z^2-z^2)+\overbar{\delta_3}(|z|^2z^2-\z^2).
\end{equation*}
Now, let $$F(z):=P_{3,\D}(z+1,\z+1)-P_{3,\D}(z,\z)-|z|^{6}P_{3,\D}\left(\frac{z+1}{z},\frac{\z+1}{\z}\right)$$ and by Proposition \ref{PPolyforZ[i]prop} we know that $F(z)=0$ for any $z\in \mathbb{C}$. Since $F(2)=-120\delta_1-240\delta_3$ we get that $\delta_1=-2\delta_3$, which also implies that $\delta_3=\overbar{\delta_3}$. So, 
\begin{equation*}
P_{3,\D}(z,\z)=\alpha_{3,\D}(|z|^6-1)-2\delta_3(|z|^4-|z|^2)+\delta_3(|z|^2\z^2+|z|^2z^2-z^2-\z^2).
\end{equation*}
Now, we let 
\begin{align*}
G(z):=&P_{3,\D}(z,\z)-P_{3,\D}(z+\omega,\z+\bar{\omega})-|\om z+1|^{6}P_{3,\D}\left(\frac{z}{\om z+1},\frac{\z}{\omega\z+1}\right)\\
-&|z+\omega|^{6}P_{3,\D}\left(\frac{\om z+1}{z+\omega},\frac{\omega\z+1}{\z+\om}\right),
\end{align*}
and again by Proposition \ref{PPolyforZ[i]prop} part (5) we know that $G(z)=0$ for any $z\in \mathbb{C}$. Finally, since $G(2)=294\delta_3$ we see that $\delta_3=0$ and thus $$P_{3,\D}(z,\z)=\alpha_{3,\D}(|z|^6-1)$$ which completes the proof of the first assertion. 

As for the second claim, we know that 
\begin{equation*}
P_{5,\D}(z,\z)=\sum_{k_1+k_2+k_3+k_4=5}\binom{5}{k_1,k_2,k_3,k_4}\left(\sum_{\substack{N(b)-ac=\D \\c<0<a}}a^{k_1}b^{k_2}\b^{k_3}c^{k_4}\right)|z|^{2k_1}z^{k_2}\z^{k_3}
\end{equation*}
where $2|(k_{2}+k_{3})$ and $k_2+2k_3\equiv 0$ (mod 3). Thus, 
\begin{equation*}
P_{5,\D}(z,\z)=\alpha_{5,\D}|z|^{10}+\delta_1|z|^8+\delta_2|z|^6+\delta_3|z|^4+\delta_4|z|^2-\alpha_{5,\D}
\end{equation*}
where $\delta_1,\delta_2,\delta_3,\delta_4 \in \mathbb{Z}.$ Since $P_{5,\D}|(S+\1)=0$, we have that  $\delta_1=-\delta_4$ and $\delta_2=-\delta_3 $. So, 
\begin{equation*}
P_{5,\D}(z,\z)=\alpha_{5,\D}(|z|^{10}-1)+\delta_1(|z|^8-|z|^2)+\delta_2(|z|^6-|z|^4).
\end{equation*}
Finally, we again let $$F(z):=P_{5,\D}(z+1,\z+1)-P_{5,\D}(z,\z)-|z|^{10}P_{5,\D}\left(\frac{z+1}{z},\frac{\z+1}{\z}\right)$$ and by Proposition \ref{PPolyforZ[i]prop} we know that $F(z)=0$ for any $z\in \mathbb{C}$. Evaluating $F$ at $z=2$ and $z=2i$ gives the system of equations 
\begin{align*}
-17640\delta_1-5880\delta_2&=0\\
-852\delta_1-348\delta_2&=0 
\end{align*} which has $\delta_1=\delta_2=0$ as solutions. Thus,
 $$P_{5,\D}(z,\z)=\alpha_{5,\D}(|z|^{10}-1)$$ which concludes the proof. 
\end{proof}

\begin{thm}\label{H4d=w4d} Suppose that  $H_{3,\D}$ and $H_{5,\D}$ are continuous functions. Then,	
\begin{enumerate}
	\item $H_{3,\D}(z)$  has a constant value $\alpha_{3,\D}$ for all $z \in \mathbb{C}$ when $d=1,3,7$.\\
	\item $H_{5,\D}(z)$ has a constant value $\alpha_{5,\D}$ for all $z \in \mathbb{C}$ when $d=3$.\\
\end{enumerate}		
\end{thm}

\begin{proof}
	Let $\Theta_{3,\D}(z):=\Theta_{3,\D}(z) -\alpha_{3,\D}$. Then, the same argument as in the proof of Theorem \ref{H2d=w2d} shows that $\Theta_{3,\D}(z)=\Theta_{3,\D}(0)=0$ for all $z \in K$ and thus $H_{3,\D}(z) =\alpha_{3,\D}$ for all $z \in K$. Since $H_{3,\D}(z)$ is continuous, we conclude that $H_{3,\D}(z) =\alpha_{3,\D}$ for all $z \in \mathbb{C}$. The proof of the second claim is exactly the same.
	
\end{proof}


\section{Continuity of $\H$} \label{cont}

Proving that $\H$ is a continuous function when $k\geq 3$ is straightforward since in that case the sum defining $\H$ converges absolutely and uniformly as we show below. The argument we use to show the continuity for $k\geq 3$ does not apply for $k=1$, since the bound we use to show the convergence and uniform continuity is not sharp enough. Since the summands that appear in the sum defining $H_{1,\D}$ depend on $z$, to show the continuity of $H_{1,\D}$ a more explicit description of these summands in terms of $z$ is needed. The same issue also arises when trying to prove that $F_{2,D}$ is continuous. Zagier gives an argument for the convergence of $F_{2,D}$ and mentions that the continuity of $F_{2,D}$ can be proved in an elementary way, but he does not provide an argument. However, based on computer experiments he observes that for a given $x$ there is in fact a strong relation between the summands that appear in the sum for $F_{2,D}(x)$ and the continued fraction expansion of $x$. More specifically, he observes that the summands which appear in the sum defining $F_{2,D}(x)$ for a given $x$ belong to a union of two lists whose elements tend to zero exponentially quickly, and each summand in each list is obtained from the previous one by applying an element of $\textbf{SL}(2,\mathbb{Z})$ coming from the continued fraction expansion of $x$.  This observation was shown to be true by Bengoechea in \cite{Be}, and as a result of this she obtains a direct proof of the continuity of $F_{2,D}$ by showing that $F_{2,D}$ converges exponentially quickly. To prove the continuity of $H_{1,\D}$ we use Bengoechea's methods but the difficulty with that approach is finding a continued fraction algorithm for complex numbers that is similar to the classical continued fraction (i.e., nearest integer continued fraction) expansion of real numbers which Bengoechea uses in her paper. Such an algorithm was developed by Hurwitz \cite{Hur} and has been studied recently by a number of different people (see \cite{Wie,DN,DN2,DN3,Hin}, and other references cited therein).

\subsection{Continuity of $\H$ for $k\geq 3$ } 

\begin{thm}\label{contk>4}
	For $k \geq 3$, $\H$ is a continuous function.
\end{thm}
\begin{proof}
	We can write $\H$ as:
	\begin{equation*}
	H_{k,\D}(z)=\sum_{a<0}\left(\sum_{N(b)-ac=\D}\max \left(0,(a|z|^2+bz+\b\z+c)^{k}\right)\right).
	\end{equation*}
The identity $\D=N(\b+az)+|a|(a|z|^2+bz+\b\z+c),$ and the inequalities $a<0<a|z|^2+bz+\b\z+c$ together imply that $N(\b+az)<\sqrt{\D}$. So for each $z\in \mathbb{C}$, $\b$ is an algebraic integer belonging to a disk of radius $\sqrt{\D}$ centered at $-2az$. On the other hand the number of algebraic integers in a disk of fixed radius is finite. Thus, for each $a$ the number of $b$'s satisfying the conditions of the sum is finite and independent of $z$ which shows that for each $a$ the number functions appearing in the inner sum is finite and independent of $z$. Also, for each $a$ the functions appearing in the inner sum are bounded by $\frac{\D}{|a|}$. So, for each $a$ and all $z$ the inner sum is the sum of finitely many continuous functions $h(z,1)^{k}=(a|z|^2+bz+\b\z+c)^{k}$ that are bounded by $\left(\frac{\D}{|a|}\right)^{k}$. Hence by the Weierstrass M-test, the sum in the definition of $H_{k,\D}$ converges absolutely and uniformly and is a continuous function for $k\geq 3$.

\end{proof}

\subsection{Nearest Integer Continued Fraction Algorithm for the Euclidean Imaginary Quadratic Fields.\\}

Let $z$ be a complex number. We denote by $\lfloor z \rceil $ the integer in $\OO_d$ that is nearest to $z$ with respect to Euclidean distance in complex plane, rounding down, in both the real and the imaginary components to break ties. The nearest integer continued fraction algorithm for the Euclidean imaginary quadratic fields proceeds by steps of the form
\begin{equation}\label{H-algorithm}
z_{0}=z, \quad \alpha_{n}=\lfloor z_{n} \rceil, \quad z_{n+1}=\frac{1}{z_{n}-\alpha_{n}} \quad 
\end{equation} 
where $\alpha_{n}=a_{n}+b_{n}\omega \in \OO_d$ and $n \geq 0$. If $z \in K$ the algorithm terminates when, as must eventually occur, $z_{n}=0$. If initially $z \notin K$ then the algorithm continues indefinitely. Once the partial quotient $\alpha_{n}$'s have been computed, the partial convergents are computed using the usual formula from the real case:
\begin{equation}\label{p-convergent}
p_{-2}=0, \quad p_{-1}=1, \quad p_{n}=\alpha_{n}p_{n-1}+p_{n-2}
\end{equation}
\begin{equation}\label{q-convergent}
q_{-2}=1, \quad q_{-1}=0, \quad q_{n}=\alpha_{n}q_{n-1}+q_{n-2}.
\end{equation}
These numbers satisfy the equation 
\begin{equation*}
p_{n+1}q_{n}-p_{n}q_{n+1}=(-1)^{n},
\end{equation*}
and $\displaystyle\lim_{n \to \infty}\frac{p_n}{q_n}=z$. 

Similar to the real case, let us also define the numbers $\delta_{n}$ associated to continued fraction expansion of $z$. The numbers $\delta_{n}$ ($n \geq -1$) are defined inductively by  
\begin{equation}\label{deltas-Hurwitz}
\delta_{-1}=z, \quad \delta_0=1, \quad \delta_{n+1}=\delta_{n-1}-\alpha_n\delta_{n},
\end{equation} and satisfy 
\begin{equation}
\frac{\delta_{n}}{\delta_{n-1}}=\frac{1}{z_n}=z_{n-1}-\lfloor z_{n-1} \rceil.
\end{equation}
If $z \in K$, then $z_{}=\displaystyle\frac{p_n}{q_n}$ for some $n$ and the recurrence stops with $\delta_{n+1}=0$. If $z \notin K$, then $\delta_{n}$'s decay to zero exponentially quickly in $n$.

Finally, let us define the set of matrices associated to the continued fraction of $z$
\begin{gather*}
\hat{\Gamma}_d(z):=\left\{\gamma_n \in \hat{\Gamma}_d \,:\, \gamma_{n}=\left(\begin{array}{cc}
q_{n-2} & -p_{n-2} \\ 
-q_{n-1} & p_{n-1}
\end{array}\right) \right\}
\end{gather*} where $\hat{\Gamma}_d=\textbf{PGL}(2,\OO_d)$.

\subsection{Continuity of $H_{1,\D}$}
For a fixed $z \in \mathbb{C}$  denote the set of functions appearing in the sum for $H_{1,\D}(z)$ by 
$$ \Omega_{\D}(z)=\left\{h \in \mathcal{H}(\mathcal{O}_d,-\D)\ \big| \ a_h<0<h(z,1) \right\}.$$ 
Let $\mathcal{A}$ be an equivalence class in $\mathcal{H}(\mathcal{O}_d,-\D)/\hat{\Gamma}_d$, and define 
\begin{align}
	\mathcal{A}(z)&:=\{h\in \mathcal{A}\ \big|\ a_h<0<h(z,1)\},\\
	\hat{\mathcal{A}}&:=\{h\in \mathcal{A}\ \big|\ h(0,1)<0<a_h\}.
\end{align}
Then, the sum defining $H_{1,\D}$ can be written as 
\begin{equation}\label{myfunc2}
H_{1,\D}(z)=\sum_{\mathcal{A}\in \mathcal{H}(\mathcal{O}_d,-\D)/\hat{\Gamma}_d}\left(\sum_{h \in \mathcal{A}(z)} h(z,1)\right).
\end{equation}
The next proposition gives a description of $\mathcal{A}(z)$, which is enough to prove the continuity of $H_{1,\D}$.
\begin{prop}\label{surjection}
Let $z \in \mathbb{C}$ be fixed, $\mathcal{A}$ be an equivalence class in $\mathcal{H}(\mathcal{O}_d,-\D)$ and
\begin{gather*}
\mathcal{B}(z)=\left\{(f,\gamma)\in \hat{\mathcal{A}}\times \hat{\Gamma}_d(z)\ |\ f(\gamma(\infty),1)<0<f(\gamma(z),1) \right\},
\end{gather*} where for $\gamma=\begin{pmatrix}
r&s\\t&v
\end{pmatrix}$ we let	
\begin{gather*}
\gamma (z):= \frac{rz+s}{tz+v} \quad \text{and} \quad \gamma (\infty):=\lim_{z \to \infty}\gamma{(z)}=\frac{r}{t}.
\end{gather*}	
	 Then, the map $\Phi:\mathcal{B}(z)\rightarrow \mathcal{A}(z)$ defined by $\Phi(f,\gamma)=\overbar{\gamma}(f)$ is a surjection.
\end{prop} 

\begin{proof}
	First we check that the map is well defined. Let $$f= \begin{pmatrix}
	a&b\\ \b&c
	\end{pmatrix}\in \hat{\mathcal{A}} {\ \  \text{and}\ \ }   \gamma=\begin{pmatrix}
	r&s\\t&v
	\end{pmatrix}\in \hat{\Gamma}_d(z).$$ Then 
	$$\overbar{\gamma}(f)=\gamma^t f \overbar{\gamma}=\begin{pmatrix}
	r&t\\s&v
	\end{pmatrix}\begin{pmatrix}
	a&b\\ \b&c
	\end{pmatrix}\begin{pmatrix}
	\bar{r}&\bar{s}\\\bar{t}&\bar{v}
	\end{pmatrix},$$
	 and we see that $$a_{\overbar{\gamma}(f)} =a|r|^2+br\bar{t}+\b\bar{r}t+c|t|^2,$$
	which is equal to $$|t|^2f\left(\frac{r}{t},1\right)=|t|^2f(\gamma(\infty),1).$$ But by assumption $f(\gamma(\infty),1)<0$, thus $a_{\overbar{\gamma}(f)}<0.$ Also, $f(\gamma(z),1)>0$ implies that $\overbar{\gamma}(f)(z,1)>0$ since $\displaystyle \overbar{\gamma}(f)(z,1)=|tz+v|^2f(\gamma(z),1)$. So, $\Phi$ is well defined.\\

\noindent To show that $\Phi$ is surjective; let $h \in \mathcal{A}(z)$ and consider the identity $$-\D(h)=N(\b_h+a_hZ)+|a_h|h(Z,1)$$ which holds for any $Z \in \mathbb{C}$. Using this identity we see that the set of zeros of $h(Z,1)$ satisfy the equation 
\begin{equation}\label{zeroset}
	\D=N(\b_h+a_hZ)=a_h^2\left|Z+\frac{\b_h}{a_h}\right|^2,
\end{equation} which describes a circle in the complex plane centered at $-\frac{\b_h}{a_h}$ with radius $\frac{\sqrt{\D}}{|a_h|}.$\\
Let $B$ be the open disk centered at $-\frac{\b_h}{a_h}$ with radius $\frac{\sqrt{\D}}{|a_h|}$. From equation (\ref{zeroset}) it is clear that $N(\b_h+a_hZ)<\D$ if $Z \in B$ and $N(\b_h+a_hZ)>\D$ if $Z \notin B$. Hence $h(Z,1)>0$ if $Z \in B$ and $h(Z,1)<0$ if $Z \notin B$ and in particular $z\in B$, since $h \in \mathcal{A}(z)$ implies that $h(z,1)>0$. Let $p_n$ and $q_n$ be the partial convergents of $z$ and because $\displaystyle\frac{p_n}{q_n} \rightarrow z$ as $n \rightarrow \infty$, we can find a positive integer $n$ large enough so that $\displaystyle\frac{p_{n-1}}{q_{n-1}} \notin B$ but $\displaystyle\frac{p_{n}}{q_{n}}, \frac{p_{n+1}}{q_{n+1}},\frac{p_{n+2}}{q_{n+2}}, \ldots \in B$. If no such $n$ exists, we set $n=0$. In both cases we have 
	\begin{equation*}
		h\left(\frac{p_{n-1}}{q_{n-1}},1\right)<0<h\left(\frac{p_{n}}{q_{n}},1\right)
	\end{equation*}
where we observe that $h(\infty,1)<0$ when $n=0$, since $a_h<0$.
	Let $$\gamma=\gamma_{n+1}=\left(\begin{array}{cc}
	q_{n-1} & -p_{n-1} \\ 
	-q_{n} & p_{n}
	\end{array}\right) \in \Gamma_d(z) \implies \gamma^{-1}= \left(\begin{array}{cc}
	p_{n} & p_{n-1} \\ 
	q_{n} & q_{n-1}
	\end{array}\right),$$
	and $$ f=\overline{\gamma^{-1}}(h)=\begin{pmatrix}
	p_n&q_n\\p_{n-1}&q_{n-1}
	\end{pmatrix}\begin{pmatrix}
	a&b\\ \b&c
	\end{pmatrix}\begin{pmatrix}
	\bar{p}_n&\bar{p}_{n-1}\\\bar{q}_n&\bar{q}_{n-1}
	\end{pmatrix} .$$ Then, $\displaystyle f(0,1)=|q_{n-1}|^2h\left(\frac{p_{n-1}}{q_{n-1}},1\right)<0$ and $\displaystyle a_f=|q_n|^2h\left(\frac{p_{n}}{q_{n}},1\right)>0$. Thus $f \in \hat{\mathcal{A}}$. Moreover $\overbar{\gamma}(f)=h$ which completes the proof of surjectivity.
\end{proof}

\begin{thm}
	$H_{1,\D}$ is a continuous function for all $z\in \mathbb{C}$. 
\end{thm}

\begin{proof}
	From the equation (\ref{myfunc2}) we see that $H_{1,\D}(z)$ is the sum of the sums $\displaystyle\sum_{h \in \mathcal{A}(z)} h(z,1)$ over all the $\hat{\Gamma}_d$-equivalence classes in $\mathcal{H}(\mathcal{O}_d,-\D)$. By Proposition \ref{surjection}, we have that 
	\begin{equation}\label{sumoverequiv}
	\sum_{h \in \mathcal{A}(z)} h(z,1) \leq \sum_{f\in \mathcal{A}(0)}\sum_{\substack{\gamma \in \hat{\Gamma}_d(z) \\f(\gamma(\infty),1)<0\\f(\gamma(z),1)>0}}\overbar{\gamma}(f)(z,1).
	\end{equation}
	There are finitely many $f$ in $\mathcal{A}(0)$. So, the outer sum on the right side of inequality (\ref{sumoverequiv}) is finite. Furthermore, the summands in the inner sum are of the form $$\overbar{\gamma}_{n}f(z,1)= a_f|\bar{\delta}_{n-1}|^2+b_f\delta_{n-1}\bar{\delta}_{n}+\b_f \bar{\delta}_{n-1}\delta_{n}+c_f|\delta_n|^2$$ for some $n \geq 1$. The sequence $\delta_{n}=|p_{n-1}-q_{n-1}z|$ stops if $z \in K$ and decreases to 0 exponentially quickly in $n$  if $z \notin K$. So, the summands of the double sum on the right side of the equation (\ref{sumoverequiv}) converge to $0$ exponentially quickly in $n$, which implies that $H_{1,\D}(z)$ converges exponentially quickly in $n$ for any $z\in \mathbb{C}$. Thus $H_{1,\D}(z)$ is a continuous function for all $z \in \mathbb{C}$.
\end{proof}

\section{The average value of $\H$ and Cohen-Zagier type formulas}\label{averagevalue}

As we saw in Proposition \ref{H-properties} part (3), the function $\H$ is $\OO_d$-invariant, and so it has a well-defined average value which we denote by $\langle \H(z) \rangle_{av}$. We let $ \mathcal{P}$ denote the fundamental parallelogram for the lattice $\OO_d$ and $|\mathcal{P}|$ denote the Euclidean area enclosed by $\mathcal{P}$. Since $\left\{1,\omega=\frac{d_K+\sqrt{d_K}}{2}\right\}$ is a $\mathbb{Z}$-basis of $\OO_d$, we see that $|\mathcal{P}|=\frac{\sqrt{|d_K|}}{2}$. We also define the following subgroup of $\textbf{PSL}(2,\OO_d)$
$$\tilde{\Gamma}_d:=\left\{\begin{pmatrix}
1&\lambda\\0&1
\end{pmatrix}\ \bigg| \ \lambda \in \OO_d \right\}.$$ Notice that the set of binary Hermitian forms corresponding to Hermitian matrices in the orbit of $h\in \mathcal{H}(\mathcal{O}_d,-\D)$ under the action of $\tilde{\Gamma}_d$ is given by $$\tilde{\Gamma}_d(h)=\left\{h(z+\lambda,1)\ | \ \lambda \in \OO_d \right\}.$$ Finally, we define the following zeta function. For $\D, n \in \mathbb{Z} $, we set 
$$r(\D,n):=\#\{\beta \in \OO_d/n\OO_d \ |\ \beta\bar{\beta}+\D\equiv0\ (\text{mod}\ n)\}$$ where we denote by $\#S$ the cardinality of a set $S$ and define
$$Z(\D,s):=\sum_{n=1}^{\infty}\frac{r(\D,n)}{n^{s+1}}.$$ This zeta function is studied by Elstrodt, Grunewald and Mennicke \cite{EGM1,EGM2} in connection with representations numbers of binary Hermitian forms with coefficients in $\OO_d$. They show that 
\begin{equation}\label{EGM-zeta}
Z(\D,s)= \begin{cases} 
\zeta_{K}(s)L(\chi_{d_K},s+1)^{-1}&\text{if }  \D=0 \\
\theta(\D,s)\zeta_{\mathbb{Q}}(s)L(\chi_{d_K},s+1)^{-1} &\text{if }  \D\neq 0 \\
\end{cases}
\end{equation}
where $\zeta_{K}(s)$ denotes the usual zeta function of $K$ and $\theta(\D,s)$ is a finite Euler product given by 

\begin{equation}\label{eulerpro}
\theta(\D,s)=\prod_{p|d_k\D}R_p(\D,p^{-1-s})
\end{equation}
with

\begin{equation*}
R_p(\D,X)=\begin{cases}
\frac{1-\left(\left(\frac{d_K}{p}\right)(pX)\right)^{t+1}}{1-\left(\frac{d_K}{p}\right)pX}&\text{for }  p\nmid d_K,\ p^t\edv \D,\\
1+\left(\frac{-|D_0|^t\D_0}{p}\right)(pX)^{t+1}&\text{for }  p\mid d_K,\ p\neq 2,\ p^t\edv \D,\\
1+\left(\frac{8}{\D_0D^t_2}\right)(2X)^{t+3}&\text{for } p=2,\  4\mid d_K,\ D_{1}\equiv 2(8),\ 2^t\edv \D,\\
1-\left(\frac{-8}{\D_0D^t_2}\right)(2X)^{t+3}&\text{for } p=2,\  4\mid d_K,\ D_{1}\equiv 6(8),\ 2^t\edv \D,\\
1-\left(\frac{-4}{\D_0D^t_2}\right)(2X)^{t+2}&\text{for } p=2,\  4\mid d_K,\ D_{1}\equiv 3\ \text{or}\ 7(8),\ 2^t\edv \D,\\
\end{cases}
\end{equation*}
where $D_{0}:=d_K/p$, $\D_0:=p^{-t}\D$, and  where for $d_K\equiv0\pmod4$ 

$$D_{1}:=\frac{d_K}{4}, \quad \quad D_{2}:=\begin{cases}
-\frac{D_{1}}{2}&\text{if } D_{1}\equiv2\pmod4,\\
\frac{1-D_{1}}{2}&\text{if } D_{1}\equiv3\pmod4.
\end{cases}$$

\begin{thm}\label{avg} Let $k$ and $\D$ be as in Definition \ref{myfuncdef}. Then, the average value of $\H$ is given by
\begin{equation}
\langle \H(z) \rangle_{av}=\frac{2\pi\D^{k+1}}{(k+1)\sqrt{|d_K|}}\,Z(-\D,k+1).
\end{equation}	

\end{thm}

\begin{proof}
We have 	
\begin{equation*}
\langle \H(z) \rangle_{av}=\frac{1}{|\P|}\int_{\P}\sum_{\substack{h\in \mathcal{H}(\mathcal{O}_d,-\D)\\a_h<0}}\max(0,h(z,1)^{k})\,dx\,dy
\end{equation*}
Since $\H$ is continuous, it converges uniformly on $\P$. Thus, we can interchange the integral and the sum to get	
\begin{align*}
\langle \H(z) \rangle_{av}&=\frac{1}{|\P|}\sum_{\substack{h\in \mathcal{H}(\mathcal{O}_d,-\D)\\a_h<0}}\int_{\P}\max(0,h(z,1)^{k})\,dx\,dy\\
&=\frac{1}{|\P|}\sum_{\substack{h\in \mathcal{H}(\mathcal{O}_d,-\D)/\tilde{\Gamma}_d\\a_h<0}}\left(\sum_{\lambda \in \OO_d}\int_{\P}\max(0,h(z+\lambda,1)^{k})\right)\,dx\,dy\\
&=\frac{1}{|\P|}\sum_{\substack{h\in \mathcal{H}(\mathcal{O}_d,-\D)/\tilde{\Gamma}_d\\a_h<0}}\int_{\mathbb{C}}\max(0,h(z,1)^{k})\,dx\,dy\\
&=\frac{1}{|\P|}\sum_{\substack{h\in \mathcal{H}(\mathcal{O}_d,-\D)/\tilde{\Gamma}_d\\a_h<0}}\int_{-\infty}^{+\infty}\int_{-\infty}^{+\infty}\max(0,h(z,1)^{k})\,dx\,dy.\\
\end{align*}
Making the substitution $z=\frac{-\b_h+w\sqrt{\D}}{a_h}$ where $w=u+iv \in \mathbb{C}$ produces 
\begin{align*}
\langle \H(z) \rangle_{av}&=\frac{1}{|\P|}\sum_{\substack{h\in \mathcal{H}(\mathcal{O}_d,-\D)/\tilde{\Gamma}_d\\a_h<0}}\int_{-\infty}^{+\infty}\int_{-\infty}^{+\infty}\max\left(0,\left(\frac{(w\w-1)\D}{a_h}\right)^{k}\right)\frac{\D}{a_h^2}\,du\,dv\\
&=\frac{\D^{k+1}}{|\P|}\sum_{\substack{h\in \mathcal{H}(\mathcal{O}_d,-\D)/\tilde{\Gamma}_d\\a_h<0}}\frac{1}{|a_h|^{k+2}}\int_{-\infty}^{+\infty}\int_{-\infty}^{+\infty}\max\left(0,(1-w\w)^{k}\right)\,du\,dv\\
&=\frac{\D^{k+1}}{|\P|}\sum_{\substack{h\in \mathcal{H}(\mathcal{O}_d,-\D)/\tilde{\Gamma}_d\\a_h<0}}\frac{1}{|a_h|^{k+2}}\iint\limits_{R}(1-u^2-v^2)^{k}\,du\,dv\\
\end{align*}
where $R$ is the unit circle centered at the origin. The double integral can be evaluated using polar coordinates to get $$\iint\limits_{R}(1-u^2-v^2)^{k}\,du\,dv=\int_{0}^{2\pi}\int_{0}^{1}(1-r^2)^{k}r\,dr\,d\theta=\frac{\pi}{k+1}.$$ So, we obtain 
\begin{align*}
\langle \H(z) \rangle_{av}&=\frac{\D^{k+1}}{|\P|}\sum_{\substack{h\in \mathcal{H}(\mathcal{O}_d,-\D)/\tilde{\Gamma}_d\\a_h<0}}\frac{1}{|a_h|^{k+2}}\frac{\pi}{k+1}\\
&=\frac{2\pi\D^{k+1}}{(k+1)\sqrt{|d_K|}}\sum_{\substack{h\in \mathcal{H}(\mathcal{O}_d,-\D)/\tilde{\Gamma}_d\\a_h<0}}\frac{1}{|a_h|^{k+2}}
\end{align*}
Notice that, the action of $\tilde{\Gamma}_d$ on $h=\begin{psmallmatrix}a & b\\ \b & c \end{psmallmatrix} $ corresponds to simply shifting $b$ and $\b$ by multiples of $a$ in $\OO_d$ since 
$$\gamma(h)=\begin{pmatrix}
a&a\lambda+b\\a\overbar{\lambda}+\b&a\lambda\overbar{\lambda}+b\overbar{\lambda}+\b\lambda+c
\end{pmatrix}$$
where $\gamma=\begin{psmallmatrix}1 & \lambda\\ 0 & 1 \end{psmallmatrix}.$ Since $a_h$ is fixed under the action of $\tilde{\Gamma}_d$,  the number of equivalence classes in $\mathcal{H}(\mathcal{O}_d,-\D)/\tilde{\Gamma}_d$ with $a_h=a$ is precisely $$\#\{b \in \OO_d/a\OO_d \ |\ b\b\equiv \D\ (\text{mod}\ a)\}=r(-\D,a).$$ Thus, we obtain 
\begin{align*}
\langle \H(z) \rangle_{av}&=\frac{2\pi\D^{k+1}}{(k+1)\sqrt{|d_K|}}\sum_{a=1}^{\infty}\frac{r(-\D,a)}{a^{k+2}}.
\end{align*}
\end{proof}

As an application of Theorem \ref{avg} we get the following explicit formulas for the special values.

\begin{cor}\label{Lformulas} \leavevmode	
\begin{enumerate}
	\item Let $d=1,2,3,7,11$ and $d_K$ be the discriminant of $K=\mathbb{Q}(\sqrt{-d})$. Then,
\begin{align}
L(\chi_{d_K},3)=&\frac{\pi^3\D^2}{6\alpha_{2,\D}\sqrt{|d_K|}}\ \theta(\D,2)\label{lformula1},\\
	L(\chi_{d_K},-2)=&\frac{-|d_K|^2\D^2}{12\alpha_{2,\D}}\ \theta(\D,2)\label{lformula2}.
\end{align}
\item Let $d=1,3,7$ and $d_K$ be the discriminant of $K=\mathbb{Q}(\sqrt{-d})$. Then,
\begin{align}
L(\chi_{d_K},5)=&\frac{\pi^5\D^4}{180\alpha_{4,\D}\sqrt{|d_K|}}\ \theta(\D,4),\\
L(\chi_{d_K},-4)=&\frac{|d_K|^4\D^4}{120\alpha_{4,\D}}\ \theta(\D,4).
\end{align}

\item Let $d=3$ and $d_K$ be the discriminant of $K=\mathbb{Q}(\sqrt{-d})$. Then,
\begin{align}
L(\chi_{d_K},7)=&\frac{\pi^7\D^6}{2835\alpha_{6,\D}\sqrt{3}}\ \theta(\D,6),\\
L(\chi_{d_K},-6)=&\frac{-81\D^6}{28\alpha_{6,\D}}\ \theta(\D,6).
\end{align}
\end{enumerate} 

\end{cor}

\begin{proof}
Since $\H$ is constant for $k=1$ and since $\zeta_{\mathbb{Q}}(2)=\frac{\pi^2}{6}$, we immediately get the formula in the equation (\ref{lformula1}) using Theorem \ref{avg} and the equation (\ref{EGM-zeta}). To get the formula in the equation (\ref{lformula2}) we use the functional equation 
$$L(\chi_{d_K},1-s)=2|d_K|^{s-1/2}(2\pi)^{-s}\Gamma(s)\sin\left(\frac{s\pi}{2}\right)L(\chi_{d_K},s).$$

The other formulas are obtained similarly using $\zeta_{\mathbb{Q}}(4)=\frac{\pi^4}{90}$ and $\zeta_{\mathbb{Q}}(6)=\frac{\pi^6}{945}$.

\end{proof}

\section{Cocycle property of $\H$}

Another important property of $\H$ is that it can be used to construct a nontrivial cocycle belonging to cohomology group $H^{1}(\Gamma_d,V_{k,k}(\CC))$. The groups $\Gamma=\psl2$ where $\OO$ is the ring of integers of any imaginary quadratic field are called Bianchi groups and their cohomology groups are fundamental to the study of Bianchi modular forms (see \cite{Sen1}).

Let us now define $H^{1}(\Gamma_d,V_{n,n})$.  A map $f:\Gamma_d\to V_{n,n} $ is called a $1$-cocycle if it satisfies 
\begin{equation}\label{cocycle-condition}
	f(\gamma_1\gamma_{2})=f(\gamma_{1})|\gamma_{2}+f(\gamma_{2}) 
\end{equation}
for all $\gamma_{1},\gamma_{2}\in \Gamma_d$ and where $|\gamma_2$ denotes the action of $\gamma_{2}$ on $f(\gamma_1)$ as defined in equation (\ref{twovariable}). For a fixed $P\in V_{n,n}$ the map 
\begin{equation}\label{coboundary-condition}
	\gamma\mapsto P|\gamma-P
\end{equation}
is a 1-cocycle and is called a 1-coboundary corresponding to $P$. The set of 1-cocycles denoted by $C(\Gamma_d,V_{n,n})$ and the set of 1-coboundaries denoted by $B(\Gamma_d,V_{n,n})$ are $\CC$-vector spaces. We have that $B(\Gamma_d,V_{n,n}) \subset C(\Gamma_d,V_{n,n})$, and the (first) cohomology group $H^{1}(\Gamma_d,V_{n,n})$ of $\Gamma_d$ with coefficients in $V_{n,n}$ is defined to be
\begin{equation}\label{cohomology-group}
	H^{1}(\Gamma_d,V_{n,n})=\frac{C(\Gamma_d,V_{n,n})}{B(\Gamma_d,V_{n,n})}.
\end{equation}

In general, if $G$ is a finitely presented group and $M$ is an $RG$-module for a commutative ring $R$ then $H^{1}(G,M)$ can be computed using the presentation of the group. An illustration of how this is carried out is given in \cite{Sen} through an example (see also \cite{FGT}). Since Bianchi groups are finitely presented this method is used in \cite{Sen,FGT} to study the structure of $H^{1}(\Gamma,V_{n,n})$. Using the same method we will define a certain space of cocycles in $H^{1}(\Gamma_d,V_{k,k})$ that contains the cocycle constructed using $\H$.

Recall that our function $\H$ is defined when $\OO_d$ is the ring of integers of an Euclidean imaginary quadratic field, i.e when $d=1,2,3,7,11$. So, we fix $\omega=i$ in $\OO_{1}$, $\omega=i\sqrt{2}$ in $\OO_{2}$, $\omega=\frac{-1+i\sqrt{3}}{2}$ in $\OO_{3}$, $\omega=\frac{1+i\sqrt{7}}{2}$ in $\OO_{7}$, $\omega=\frac{1+i\sqrt{11}}{2}$ in $\OO_{11}$. We also let 
\begin{gather*}
	S=\begin{pmatrix}
		0&-1\\ 1&0 
	\end{pmatrix},\hskip 0.1in T=\begin{pmatrix}
		1&1\\ 0&1 
	\end{pmatrix},\hskip 0.1in T_{\omega}=\begin{pmatrix}
		1&\omega\\ 0&1 
	\end{pmatrix}.
\end{gather*}

\subsection{ The case $d=1$} 
The following presentation of $\Gamma_{1}$ is given in \cite{Fine}
\begin{gather*}
	\Gamma_{1}=\langle\, S,T,T_{\omega}, L\,\big|\,  S^2=L^2=(SL)^2=(TL)^2=(T_{\omega}L)^2=(ST)^3\\ 
	=(T_{\omega}SL)^3=[T,T_{\omega}]=\1\, \rangle
\end{gather*}
where $L=\begin{pmatrix}
i&0\\ 0&-i 
\end{pmatrix}$, and $[T,T_{\omega}]$ is the commutator of $T$ and $T_\omega$.


We are interested in the cocycles belonging to the subspace $C_{p}(\Gamma_1,V_{k,k})$ of $C(\Gamma_1,V_{k,k})$ defined as
\begin{gather*}
	C_{p}(\Gamma_1,V_{k,k})=\{f \in C(\Gamma_1,V_{k,k})\, |\, f(T)=f(T_{\omega})=f(L)=0 \}.
\end{gather*}
The analogous space of cocycles in the classical case of $\textbf{PSL}(2,\mathbb{Z})$ is called the space of parabolic 1-cocycles and hence the subscript $p$ in the notation $C_{p}(\Gamma_1,V_{k,k}).$ The next proposition shows that $C_{p}(\Gamma_1,V_{k,k})$ can be identified with the following subspace $W_{k,k}$ of $V_{k,k}$:
\begin{align*}
	W_{k,k}&:=\{P\in V_{k,k}\,:\,P|(\1+S)=P|(\1-L)=P|(\1+U+U^2)=P|(\1+E+E^2)=0\}\\
	&=ker(\1+S)\cap ker(\1-L)\cap ker(\1+U+U^2)\cap ker(\1+E+E^2)
\end{align*} where $U=TS$ and $E=T_{\omega}SL$.

\begin{prop}\label{W1}
	The map sending $f\in C_{p}(\Gamma_1,V_{k,k}) $ to $P_S \in V_{k,k}$ is a $\CC$-vector space isomorphism from $C_{p}(\Gamma_1,V_{k,k})$ to $W_{k,k}$.
\end{prop}

\begin{proof}
	The proof is similar to that of the classical case which can be found in \cite[ Lemma 11.8.9]{CohSt}). Let $f\in C_{p}(\Gamma_1,V_{k,k})$ such that $f(\gamma)=P_{\gamma}\in V_{k,k}$ and assume that $f\mapsto f(S)=P_S$. Since $S^2=\1$, the cocycle property of $f$ gives
	\begin{gather*}
		0=f(S^2)=f(S)|S+f(S) \implies P_S|S+P_S=P_S|(\1+S)=0.
	\end{gather*}
	Again using the cocycle property of $f$ and the fact that $f(L)=0$ leads to 
	\begin{align*}
		f(SL)&=f(S)|L+f(L)=P_S|L\\
		f(LS)&=f(L)|S+f(S)=P_S
	\end{align*}
	But $SL=LS$ and hence $P_S|(\1-L)=0$.
	
	The relation $(ST)^3=\1$ gives
	\begin{gather*}
		f(ST)|(ST)^2+f(ST)|(ST)+f(ST)=0,
	\end{gather*} 
	and since $f(ST)=f(S)|T+f(T)=P_S|T$ we obtain
	\begin{align*}
		(P_S|T)|(ST)^2+(P_S|T)|ST+P_S|T&=0\\
		P_S|TSTST+P_S|TST+P_S|T&=0\\
		(P_S|U^2+P_S|U+P_S)|T&=0.
	\end{align*}
	This shows that $P_S|(\1+U+U^2)\in ker(T)$ implying that $P_S|(\1+U+U^2)$ is the zero polynomial. 
	
	Similarly, $P_S|(\1+E+E^2)=0$ follows from the fact that $f(T_{\omega}SL)=P_S|L$ and $P_S|L=P_S$ which proves that $P_S\in W_{k,k}$.
	
	If $P_S=0$, then $f$ vanishes on all the generators of $\Gamma_1$ and the cocycle property of $f$ implies that $f(\gamma)=P_\gamma=0$ for all $\gamma \in \Gamma_1$. Thus the map $f\mapsto f(S)=P_S$ is injective.
	
	Finally, we need to show that the map $f\mapsto f(S)=P_S$ is surjective. Let $P\in W_{k,k}$  and set $f(S)=P|S$ for $f\in C_{p}(\Gamma_1,V_{k,k})$. Then, it is easily seen that $f$ agrees with the defining relations of $\Gamma_1$, i.e., $f(S^2)=0$ and the same holds for the other defining relations. The surjectivity then follows from the fact that 
	\begin{gather*}
		S^2=L^2=(SL)^2=(TL)^2=(T_{\omega}L)^2=(ST)^3=(T_{\omega}SL)^3=[T,T_{\omega}]=\1\
	\end{gather*}
	generates all relations for $\Gamma_1$.
\end{proof}

Let us now discuss how the function $\H$ can be used to obtain a nontrivial element of $C_p(\Gamma_{1},V_{k,k})$. Recall in Proposition \ref{S-act} we showed that
\begin{gather*}
	\H|(\1-S)=P_{k,\D}(z,\z)=\sum_{\substack{N(b)-ac=\D \\ c<0<a}} \left(az\z+bz+\b\z+c\right)^{k}
\end{gather*} where $b\in \OO_1$ and $a,c \in \mathbb{Z}.$ In what follows we will show, that the polynomial $P_{k,\D}$ belongs to $W_{k,k}$ so that the map $f:\Gamma_{1} \to V_{k,k}$ defined by 
\begin{gather*}
	\gamma \mapsto (\H|(\1-S))|\gamma
\end{gather*} is a cocycle belonging to $C_p(\Gamma_{1},V_{k,k})$. Indeed, we will precisely describe the subspace of $W_{k,k}$ that $P_{k,\D}$ belongs to. We will also carry out the same kind of analysis for the other values of $d$ as well, although the description of $W_{k,k}$ and its subspace vary depending on the value of $d$.


The element $\varepsilon_1=\begin{psmallmatrix}
i&0\\ 0&1 
\end{psmallmatrix} \in \textbf{PGL}(2,\OO_{1})$ acts on $V_{k,k}$ by $P|\varepsilon_1=P(iz,\overline{i}\z)$ and splits $V_{k,k}$ into a direct sum of the spaces defined as follows: The linear operator induced by the action of $\varepsilon_1$ on $V_{k,k}$ has four eigenvalues, namely $\pm 1, \pm i$ and four eigenspaces corresponding to these eigenvalues. We let $V_{k,k}^{1},V_{k,k}^{-1},V_{k,k}^{i},V_{k,k}^{-i}$   be the eigenspaces corresponding to the eigenvalues $\pm 1,\pm i$ respectively. Then
\begin{align*}
	V_{k,k}^{1}=&\{P\in V_{k,k}\,|\, P(iz,\overline{i}\z)=P(z,\z)  \}\\
	V_{k,k}^{-1}=&\{P\in V_{k,k}\,|\, P(iz,\overline{i}\z)=-P(z,\z)  \}\\
	V_{k,k}^{i}=&\{P\in V_{k,k}\,|\, P(iz,\overline{i}\z)=iP(z,\z)  \}\\
	V_{k,k}^{-i}=&\{P\in V_{k,k}\,|\, P(iz,\overline{i}\z)=-iP(z,\z)  \}
\end{align*}
and therefore $V_{k,k}=V_{k,k}^{1}\oplus V_{k,k}^{-1}\oplus V_{k,k}^{i}\oplus V_{k,k}^{-i}$. \\

The next proposition shows that $W_{k,k}$ is stable under the action of $\varepsilon_1$.

\begin{prop}\label{W1-decomp}
	$W_{k,k}|\varepsilon_1 =W_{k,k}$ and hence 
	\begin{gather*}
		W_{k,k}=W_{k,k}^{1}\oplus W_{k,k}^{-1}\oplus W_{k,k}^{i}\oplus W_{k,k}^{-i}
	\end{gather*} where $W_{k,k}^{e}=W_{k,k}\cap V_{k,k}^{e}$ for $e=\pm1,\pm i.$ 
\end{prop} 
\begin{proof}
	Let $P\in W_{k,k}$. From $LS\varepsilon_1=\varepsilon_1 S$ we get that 
	\begin{align*}
		(P|\e_1)|(\1+S)&=P|\e_1+P|LS\e_1\\
		&=(P+P|LS)|\e_1. 
	\end{align*}	
	But, $P|L=P$ which gives 
	\begin{gather*}
		(P|\e_1)|(\1+S)=(P+P|S)|\e_1=0 \implies P|\e_1 \in ker(\1+S).
	\end{gather*}
	Because $\e_1 L=L\e_1$ we immediately get that $P|\e_1 \in ker(\1-L)$. From $\e_1 U\e_{1}^3=SE^2S$, $\e_1 U^2\e_{1}^3=SES$ and $P|S=-P$ we get that 
	\begin{align*}
		(P|\e_1)|(\1+U+U^2)&=P|\e_1+P|\e_1 U+P|\e U^2\\
		&=P|\e_1+P|SE^2S\e_1+P|SES\e_1\\
		&=(P+P|SE^2S+P|SES)|\e_1\\
		&=(-P|S-P|E^2S-P|ES)|\e_1\\
		&=(-P|-P|E^2-P|E)|S\e_1\\
		&=0
	\end{align*}
	and hence $P|\e_1 \in ker(\1+U+U^2)$.
	
	Finally, because $\e_1 E\e_{1}^3=SU^2S$ and $\e_1 E^2\e_{1}^3=SUS$ the preceding argument shows that $P|\e_1 \in ker(\1+E+E^2)$ which completes the proof of $W_{k,k}|\e_1=W_{k,k}$ so that 
	\begin{gather*}
		W_{k,k}=W_{k,k}^{1}\oplus W_{k,k}^{-1}\oplus W_{k,k}^{i}\oplus W_{k,k}^{-i}.
	\end{gather*} 
\end{proof}

\begin{prop}\label{W1+}
	The polynomial $P_{k,\D}$ belongs to the subspace $W_{k,k}^{1}$ and therefore the map $f:\Gamma_{1} \to V_{k,k}$ defined by 
	\begin{gather*}
		\gamma \mapsto (\H|(\1-S))|\gamma
	\end{gather*} is a cocycle belonging to $C_p(\Gamma_{1},V_{k,k})$.
\end{prop}
\begin{proof}
	Since $S^2=\1$ we immediately get
	\begin{gather*}
		P_{k,\D}|(\1+S)=(\H|(\1-S)|(\1+S)=0\implies P_{k,\D}\in ker(\1+S).
	\end{gather*}
	Recall from Proposition \ref{H-properties} that $\H|L=\H$ and since $SL=LS$ it follows that $P_{k,\D} \in ker(\1-L)$. Also from the same proposition we have 
	\begin{gather*}
		\H|T=\H \quad \text{and}\quad \H|T_{\omega_1}=\H.
	\end{gather*}
	Since $TS=U$ and $\H|T=\H$  notice that 
	\begin{gather*}
		P_{k,\D}=\H-\H|S=\H-(\H|T)|S=\H|(\1-U).
	\end{gather*}
	But $U^3=\1$ and so 
	\begin{gather*}
		(\H|(\1-U))|(\1+U+U^2)=\H|(\1-U+U-U^2+U^2-U^3)=0,
	\end{gather*}
	which shows that $P_{k,\D}\in ker(\1+U+U^2)$. Again, using $\H|T_{\omega}=\H$, $\H|L=\H$ and $LS=SL$ gives 
	\begin{gather*}
		P_{k,\D}=\H-\H|S=\H-(\H|T_{\omega}L)|S=\H|(\1-E).
	\end{gather*}
	Since $E^3=\1$ the preceding argument shows that $P_{k,\D}\in ker(\1+E+E^2)$ implying that $P_{k,\D}\in W_{k,k}$. Finally, by part (2) of Proposition \ref{PPolyforZ[i]prop} we have 
	\begin{gather*}
		P_{k,\D}|\e_1=P_{k,\D},
	\end{gather*}
	which concludes the proof that $P_{k,\D}\in W_{k,k}^{1} $.
\end{proof}
\newpage	
We wrote a computer program and made use of SageMath \cite{Sa} to compute the dimensions of $W_{k,k}$ and $W_{k,k}^{e}$ for $1\leq k\leq 31$. For those values of $k$ we found that $\textbf{dim}(W_{k,k}^{i})=\textbf{dim}(W_{k,k}^{-i})=0$. The data about the dimensions of $W_{k,k}$, $W_{k,k}^{1}$, and $W_{k,k}^{-1}$ for $1\leq k \leq 31$ is as follows:
\begin{table}[h]
	\centering
	\begin{tabular}{c rrrrrrrrrrrrrrrrr}

		$k$ & \vline & 1 & 3 & 5 & 7 & 9 & 11 & 13 & 15 & 17 & 19 & 21 & 23 & 25 & 27 & 29 & 31\\ 
		\hline
		$\textbf{dim}(W_{k,k}^{1})$ & \vline & 1 & 1 & 2 & 2 & 3  & 3 & 4 & 4 & 5 & 5 & 6 & 6 & 7 & 7 & 8 & 8\\
		
		$\textbf{dim}(W_{k,k}^{-1})$ & \vline & 0 & 0 & 0 & 1 & 0  & 1 & 1 & 1 & 1 & 2 & 1 & 2 & 2 & 2 & 2 & 3\\
		
		$\textbf{dim}(W_{k,k})$ & \vline & 1 & 1 & 2 & 3 & 3  & 4 & 5 & 5 & 6 & 7 & 7 & 8 & 9 & 9 & 10 & 11\\
	\end{tabular}
\end{table}\\
Since $\textbf{dim}(W_{k,k}^{i})=\textbf{dim}(W_{k,k}^{-i})=0$ for $1\leq k \leq 31$, according to Proposition \ref{W1-decomp} we must have 
\begin{gather*}
	\textbf{dim}(W_{k,k}^{1})=\textbf{dim}(W_{k,k}^{1})+\textbf{dim}(W_{k,k}^{-1}).
\end{gather*}
The data in each row of the above table was computed independent of this fact, and thankfully our data agrees with it. Based on this numerical evidence we are led to the following conjecture.
\begin{conjec}
	For odd $k\geq 1$ we have that 	
	\begin{align*}
		\textbf{dim}(W_{k,k}^{1})&=\left\lfloor \frac{k-1}{4} \right\rfloor+1,\\
		\textbf{dim}(W_{k,k})&=\left\lfloor \frac{k-1}{3} \right\rfloor+1.
	\end{align*}	
	
\end{conjec}

\hskip 1in

\subsection{ The case $d=3$} The following presentation of $\Gamma_{3}$ is given in \cite{Fine}
\begin{gather*}
	\Gamma_3=\langle\, S,T,T_{\omega}, L\,\big|\,  S^2=L^3=(SL)^2=(ST)^3=(T_{\omega}SL)^3=[T,T_\omega]=\1,\\ L^{-1}T_\omega L=T,\, \, L^{-1}TL=T^{-1}T_{\omega}^{-1} \,\rangle
\end{gather*}
where $L=\begin{pmatrix}
\omega^2&0\\ 0&\omega 
\end{pmatrix}.$ The space of cocycles vanishing on parabolic elements of $\Gamma_3$ is defined to be
\begin{gather*}
	C_{p}(\Gamma_3,V_{k,k})=\{f \in C(\Gamma_3,V_{k,k})\, |\, f(T)=f(T_{\omega})=f(L)=0 \}
\end{gather*}
which again can be identified with a subspace of $V_{k,k}$. 
\begin{prop}\label{W3} The map sending $f\in C_{p}(\Gamma_3,V_{k,k}) $ to $P_S \in V_{k,k}$ is a $\CC$-vector space isomorphism from $C_{p}(\Gamma_3,V_{k,k})$ to $W_{k,k}$ where	
	\begin{gather*}
		W_{k,k}:=ker(\1+S)\cap ker(\1-L)\cap ker(\1+U+U^2)\cap ker(\1+E+E^2),
	\end{gather*} with $U=TS$ and $E=T_{\omega}SL$.
\end{prop}
\begin{proof}
	Let $f\in C_{p}(\Gamma_3,V_{k,k})$ such that $f(\gamma)=P_{\gamma}\in V_{k,k}$ and assume that $f\mapsto f(S)=P_S$. The proof that
	\begin{gather*}
		P_S\in ker(\1+S), \hskip 0.1in P_S\in ker(\1+U+U^2), \hskip 0.1in P_S\in ker(\1+E+E^2)
	\end{gather*}
	proceeds as in the proof of Proposition \ref{W1}. We also have 
	\begin{gather*}
		(SL)^2=\1 \implies P_S|LSL+P_S|L=0.
	\end{gather*}
	But $SL=L^{-1}S$ and so
	\begin{gather*}
		P_S|LSL+P_S|L=0\implies P_S|S+P_S|L=0\implies P_S|L=P_S
	\end{gather*} which shows that $P_S\in ker(\1-L).$
	
	Finally, we do not get any new equations from the remaining relations and the rest of the proof proceeds as in the proof of Proposition \ref{W1}.
\end{proof}

Let $\e_3=\begin{psmallmatrix}
\zeta&0\\ 0&1 
\end{psmallmatrix} \in \textbf{PGL}(2,\OO_{3})$ where $\zeta=e^{\frac{\pi i}{3}}$. Then $\e_3$ acts on $V_{k,k}$ by $P|\e_3=P(\e_3 z,\overline{\e_3}\z)$. The linear operator induced by the action of $\e_3$ on $V_{k,k}$ has six eigenvalues and these eigenvalues are given by the sixth roots of unity $\zeta^j$ for $1\leq j\leq6$. We let $V_{k,k}^{\zeta^j}$ be the eigenspace corresponding to $\zeta^j$. Then
\begin{gather*}
	V_{k,k}^{\zeta^j}=\{P\in V_{k,k}\,|\, P(\zeta z,\overline{\zeta}\z)=\zeta^j P(z,\z)\}
\end{gather*}
and $\displaystyle V_{k,k}=\bigoplus_{1\leq j\leq6}V_{k,k}^{\zeta^j}$. 
The next proposition is the analog of the Proposition \ref{W1-decomp} for case of $d=3$.
\begin{prop}\label{W3-decomp}
	$W_{k,k}|\e_3 =W_{k,k}$ and hence 
	\begin{gather*}
		W_{k,k}=\bigoplus_{1\leq j\leq6}W_{k,k}^{\zeta^j}
	\end{gather*} where $W_{k,k}^{\zeta^j}=W_{k,k}\cap V_{k,k}^{\zeta^j} $ for $1\leq j\leq6$. 
\end{prop} 

\begin{proof}
	The proof is identical to that of Proposition \ref{W1-decomp} after observing the following identities
	\begin{gather*}
		LS\e_3=\e_3 S, \hskip 0.2in L\e_3=\e_3L, \hskip 0.2in \e_3 U\e_3^5=SE^2S,\\ \e_3 U^2\e_3^5=SES,\hskip 0.2in \e_e E\e_3^5=SU^2S,\hskip 0.2in \e_3 E^2\e_3^5=SUS.
	\end{gather*} 
\end{proof}

\begin{prop}
	The polynomial $P_{k,\D}$ belongs to the subspace $W_{k,k}^{1}$ and therefore the map $f:\Gamma_{3} \to V_{k,k}$ defined by 
	\begin{gather*}
		\gamma \mapsto (\H|(\1-S))|\gamma
	\end{gather*} is a cocycle belonging to $C_p(\Gamma_{3},V_{k,k})$.
\end{prop}
\begin{proof}
	The proof is identical to the proof of Proposition \ref{W1+}.
\end{proof}

The computer calculations we performed using our program for $1\leq k \leq 31$ showed that $\textbf{dim}(W_{k,k})=\textbf{dim}(W_{k,k}^{1})$ and $\textbf{dim}(W_{k,k}^{j})=0$ for $2\leq j\leq 6$.   Below is a table that gives the dimension of $ W_{k,k}^{1}$ for $1\leq k \leq 31$
\begin{table}[h]
	\centering
	\begin{tabular}{c rrrrrrrrrrrrrrrrr}
		
		$k$ & \vline & 1 & 3 & 5 & 7 & 9 & 11 & 13 & 15 & 17 & 19 & 21 & 23 & 25 & 27 & 29 & 31\\ 
		\hline
		$\textbf{dim}(W_{k,k}^{1})$ & \vline & 1 & 1 & 1 & 2 & 2  & 2 & 3 & 3 & 3 & 4 & 4 & 4 & 5 & 5 & 5 & 6\\
	\end{tabular}
\end{table}\\
Based on this numerical data we are led to the following conjecture.
\begin{conjec}
	For odd $k\geq 1$ we have that $\displaystyle\textbf{dim}(W_{k,k}^{1})=\left\lfloor \frac{k}{6} \right\rfloor+1.$	
\end{conjec}

\hskip 1in

\subsection{ The case $d=2$} The following presentation of $\Gamma_{2}$ is given in \cite{Fine}
\begin{gather*}
	\Gamma_2=\langle\, S,T,T_{\omega} \,\big|\,  S^2=(ST)^3=(T_{\omega}^{-1}ST_{\omega}S)^2=[T,T_\omega]=\1\rangle.
\end{gather*}
The space of cocycles vanishing on parabolic elements of $\Gamma_2$ is defined to be
\begin{gather*}
	C_{p}(\Gamma_2,V_{k,k})=\{f \in C(\Gamma_2,V_{k,k})\, |\, f(T)=f(T_{\omega})=0 \}
\end{gather*}
which again can be identified with a subspace of $V_{k,k}$. 
\begin{prop}\label{W2} The map sending $f\in C_{p}(\Gamma_2,V_{k,k}) $ to $P_S \in V_{k,k}$ is a $\CC$-vector space isomorphism from $C_{p}(\Gamma_2,V_{k,k})$ to $W_{k,k}$ where	
	\begin{gather*}
		W_{k,k}:=ker(\1+S)\cap ker(\1+U+U^2)\cap ker(\1+ST_\omega + T_{\omega}S+T_{\omega}^{-1}ST_{\omega}S),
	\end{gather*} with $U=TS$.
\end{prop}
\begin{proof}
	Let $f\in C_{p}(\Gamma_2,V_{k,k})$ such that $f(\gamma)=P_{\gamma}\in V_{k,k}$ and assume that $f\mapsto f(S)=P_S$. The proof that
	\begin{gather*}
		P_S\in ker(\1+S), \hskip 0.1in P_S\in ker(\1+U+U^2)
	\end{gather*}
	proceeds as in the proof of Proposition \ref{W1}. As for the proof of $P_S\in ker(\1+ST_\omega + T_{\omega}S+T_{\omega}^{-1}ST_{\omega}S)$ we observe that
	\begin{gather*}
		0=f((T_{\omega}^{-1}ST_{\omega}S)^2)=P_S|T_{\omega}ST_{\omega}^{-1}ST_{\omega}S+P_S|T_{\omega}^{-1}ST_{\omega}S+P_S|T_{\omega}S+P_S
	\end{gather*}
	and since $T_{\omega}ST_{\omega}^{-1}ST_{\omega}S=ST_{\omega}$ the result follows.	
	Finally, we do not get any new equations from the remaining relations and the rest of the proof proceeds as in the proof of Proposition \ref{W1}.
\end{proof}

Let $\e=\begin{psmallmatrix}
-1&0\\ 0&1 
\end{psmallmatrix} \in \textbf{PGL}(2,\OO_{K})$. Then $\e$ acts on $V_{k,k}$ by $P|\e=P(-z,-\z)$ and splits $V_{k,k}$ up into the direct sum of the spaces $V_{k,k}^1$ and $V_{k,k}^{-1}$ where 
\begin{gather*}
	V_{k,k}^{1}=\{P\in V_{k,k}\,|\, P(-z,-\z)=P(z,\z)\}\\
	V_{k,k}^{-1}=\{P\in V_{k,k}\,|\, P(-z,-\z)=-P(z,\z)\}
\end{gather*}
and $\displaystyle V_{k,k}=V_{k,k}^{1}\oplus V_{k,k}^{-1}$.

\begin{prop}\label{W2-decomp}
	$W_{k,k}|\e =W_{k,k}$ and hence 
	\begin{gather*}
		W_{k,k}=W_{k,k}^{1}\oplus W_{k,k}^{-1}
	\end{gather*} where $W_{k,k}^{\pm 1}=W_{k,k}\cap V_{k,k}^{\pm 1} $. 
\end{prop} 

\begin{proof}
	
	Let $P\in W_{k,k}$. Using the identities 
	\begin{gather*}
		\e S= S\e, \quad \e U=SU^2S\e
	\end{gather*}
	we immediately see that $P|\e$ belongs to $ker(\1+S)$ and  $ker(\1+U+U^2)$. 
	
	It remains to show that $P|\e \in  ker(\1+ST_\omega + T_{\omega}S+T_{\omega}^{-1}ST_{\omega}S)$. Notice that $\e T_\omega=T_\omega^{-1}\e$ and so
	\begin{gather*}
		(P|\e)|(\1+ST_\omega + T_{\omega}S+T_{\omega}^{-1}ST_{\omega}S)=(P|(\1+ST_{\omega}^{-1}+T_{\omega}^{-1}S+T_\omega S T_{\omega}^{-1}S))|\e.
	\end{gather*}
	Now observe that
	\begin{align*}
		0=&P|(\1+ST_\omega + T_{\omega}S+T_{\omega}^{-1}ST_{\omega}S)\\
		=&(P|(\1+ST_\omega + T_{\omega}S+T_{\omega}^{-1}ST_{\omega}S))|T_{\omega}^{-1}S\\
		=&P|(T_{\omega}^{-1}S+\1+T_{\omega}ST_{\omega}^{-1}S+T_{\omega}^{-1}ST_{\omega}ST_{\omega}^{-1}S),
	\end{align*}
	but $T_{\omega}^{-1}ST_{\omega}ST_{\omega}^{-1}S=ST_{\omega}^{-1}$ which shows that
	\begin{gather*}
		P|(\1+ST_{\omega}^{-1}+T_{\omega}^{-1}S+T_\omega S T_{\omega}^{-1}S)=0.
	\end{gather*} It follows that $P|\e \in ker(\1+ST_\omega + T_{\omega}S+T_{\omega}^{-1}ST_{\omega}S).$
\end{proof}

\begin{prop}
	The polynomial $P_{k,\D}$ belongs to the subspace $W_{k,k}^{1}$ and therefore the map $f:\Gamma_{2} \to V_{k,k}$ defined by 
	\begin{gather*}
		\gamma \mapsto (\H|(\1-S))|\gamma
	\end{gather*} is a cocycle belonging to $C_p(\Gamma_{2},V_{k,k})$.
\end{prop}
\begin{proof}
	The proof that $P_{k,\D} \in ker(\1+S)$ and $P_{k,\D} \in ker(\1+U+U^2)$ is same as the one given in the proof of Proposition \ref{W1+}. Let us now calculate $P_{k,\D}|(\1+ST_\omega + T_{\omega}S+T_{\omega}^{-1}ST_{\omega}S)$ which is equal to
	\begin{align*}
		=&(\H|(\1-S))|(\1+ST_\omega + T_{\omega}S+T_{\omega}^{-1}ST_{\omega}S)\\
		=&\H|(\1+ST_\omega+T_{\omega}S+T_{\omega}^{-1}ST_{\omega}S-S-T_{\omega}-ST_{\omega}S-ST_{\omega}^{-1}ST_{\omega}S)\\
		=&\H|(\1+ST_\omega+T_{\omega}S+T_{\omega}^{-1}ST_{\omega}S-S-T_{\omega}-ST_{\omega}S-T_{\omega}^{-1}ST_{\omega}).
	\end{align*}
	Since $\H|T_{\omega}=\H$ and $\H|T_{\omega}^{-1}=\H$ we see that 
	\begin{gather*}
		\H|(\1+ST_\omega+T_{\omega}S+T_{\omega}^{-1}ST_{\omega}S-S-T_{\omega}-ST_{\omega}S-T_{\omega}^{-1}ST_{\omega})=0
	\end{gather*}
	which shows that $P_{k,\D}\in ker(\1+ST_\omega + T_{\omega}S+T_{\omega}^{-1}ST_{\omega}S)$. Finally, by part (2) of Proposition \ref{PPolyforZ[i]prop} we have
	\begin{gather*}
		P_{k,\D}|\e=P_{k,\D}
	\end{gather*}
	which concludes the proof that $P_{k,\D}\in W_{k,k}^{1} $.
\end{proof}

Our data again showed that for $1\leq k \leq 31$ $\textbf{dim}(W_{k,k})=\textbf{dim}(W_{k,k}^{1})$ and $\textbf{dim}(W_{k,k}^{-1})=0$. The following table gives the dimension of $ W_{k,k}^{1}$ for $1\leq k \leq 31$.
\begin{table}[h]
	\centering
	\begin{tabular}{c rrrrrrrrrrrrrrrrr}
		
		$k$ & \vline & 1 & 3 & 5 & 7 & 9 & 11 & 13 & 15 & 17 & 19 & 21 & 23 & 25 & 27 & 29 & 31\\ 
		\hline
		$\textbf{dim}(W_{k,k}^{1})$ & \vline & 1 & 2 & 3 & 4 & 5  & 6 & 7 & 8 & 9 & 10 & 11 & 12 & 13 & 14 & 15 & 16\\
	\end{tabular}
\end{table}\\
Based on this numerical data we are led to the following conjecture.
\begin{conjec}
	For odd $k\geq 1$ we have that $\displaystyle \textbf{dim}(W_{k,k}^{1})= \frac{k+1}{2} .$	
\end{conjec}

\hskip 1in

\subsection{ The case $d=7$ }
The following presentation for $\Gamma_{7}$ is given in \cite{Fine}
\begin{gather*}
	\Gamma_7=\langle\, S,T,T_{\omega} \,\big|\,  S^2=(ST)^3=(T_{\omega}^{-1}ST_{\omega}ST)^2=[T,T_\omega]=\1\rangle.
\end{gather*}
The presentation for $\Gamma_7$ is identical to that of $\Gamma_2$ except for a minor difference where the element $T$ appears in the relation  $(T_{\omega}^{-1}ST_{\omega}ST)^2$. Consequently, we can not use any of the arguments from the $d=2$ case. Moreover, this minor difference leads to a slightly more complicated description of the kernel of the linear map associated to the relation $(T_{\omega}^{-1}ST_{\omega}ST)^2$. Nonetheless, we can still define the space of parabolic cocycles on $\Gamma_7$ which is 
\begin{gather*}
	C_{p}(\Gamma_7,V_{k,k})=\{f \in C(\Gamma_7,V_{k,k})\, |\, f(T)=f(T_{\omega})=0 \}
\end{gather*}
and show that it can be identified it with a subspace of $V_{k,k}$ as follows. 
\begin{prop}\label{W7} The map sending $f\in C_{p}(\Gamma_7,V_{k,k}) $ to $P_S \in V_{k,k}$ is a $\CC$-vector space isomorphism from $C_{p}(\Gamma_7,V_{k,k})$ to $W_{k,k}$ where	
	\begin{gather*}
		W_{k,k}:=ker(\1+S)\cap ker(\1+U+U^2)\cap ker(T+ST_\omega + T_{\omega}ST+ST_{\omega}^{-1}ST_{\omega}),
	\end{gather*} with $U=TS$.
\end{prop}
\begin{proof}
	Let $f\in C_{p}(\Gamma_7,V_{k,k})$ such that $f(\gamma)=P_{\gamma}\in V_{k,k}$ and assume that $f\mapsto f(S)=P_S$. The proof that
	\begin{gather*}
		P_S\in ker(\1+S), \hskip 0.1in P_S\in ker(\1+U+U^2)
	\end{gather*}
	proceeds as in the proof of Proposition \ref{W1}. As for the proof of $P_S\in ker(T+ST_\omega + T_{\omega}ST+ST_{\omega}^{-1}ST_{\omega})$ we observe that
	\begin{gather*}
		0=f((T_{\omega}^{-1}ST_{\omega}ST)^2)=P_S|T_{\omega}STT_{\omega}^{-1}ST_{\omega}ST+P_S|TT_{\omega}^{-1}ST_{\omega}ST+P_S|T_{\omega}ST+P_S|T
	\end{gather*}
	and since 
	\begin{gather*}
		T_{\omega}STT_{\omega}^{-1}ST_{\omega}ST=ST_{\omega}, \quad TT_{\omega}^{-1}ST_{\omega}ST=ST_{\omega}^{-1}ST_{\omega}
	\end{gather*}
	the result follows.	Finally, we do not get any new equations from the remaining relations and the rest of the proof proceeds as in the proof of Proposition \ref{W1}.
\end{proof}

$W_{k,k}$ is stable under the action of $\e=\begin{psmallmatrix}
-1&0\\ 0&1 
\end{psmallmatrix}$ and splits as in the case of $d=2$. 
\begin{prop}\label{W7-decomp}
	$W_{k,k}|\e =W_{k,k}$ and hence 
	\begin{gather*}
		W_{k,k}=W_{k,k}^{1}\oplus W_{k,k}^{-1}
	\end{gather*} where $W_{k,k}^{\pm 1}=W_{k,k}\cap V_{k,k}^{\pm 1} $. 
\end{prop} 

\begin{proof}
	
	Let $P\in W_{k,k}$. The proof that $P|\e$ belongs to $ker(\1+S)$ and  $ker(\1+U+U^2)$ is same as the case of $d=2$ since the identities 
	\begin{gather*}
		\e S= S\e, \quad \e U=SU^2S\e
	\end{gather*}
	still hold. 
	
	It remains to show that $P|\e \in  ker(T+ST_\omega + T_{\omega}ST+ST_{\omega}^{-1}ST_{\omega})$. Notice that
	\begin{gather*}
		\e T=T^{-1}\e,\quad \e T_\omega=T_\omega^{-1}\e,
	\end{gather*}	
	and so
	\begin{gather*}
		(P|\e)|(T+ST_\omega + T_{\omega}ST+ST_{\omega}^{-1}ST_{\omega})=(P|(T^{-1}+ST_\omega^{-1} + T_{\omega}^{-1}ST^{-1}+ST_{\omega}ST_{\omega}^{-1}))|\e.
	\end{gather*}
	Now, since 
	\begin{gather*}
		S^2=\1,\quad T_\omega^{-1}T^{-1}=T^{-1}T_\omega^{-1},\quad \text{and}\quad P|S=-P,
	\end{gather*}
	we have
	\begin{align*}
		0=&(-P)|(T+ST_\omega + T_{\omega}ST+ST_{\omega}^{-1}ST_{\omega})\\
		=&((-P)|(T+ST_\omega + T_{\omega}ST+ST_{\omega}^{-1}ST_{\omega}))|T_{\omega}^{-1}T^{-1}\\
		=&(-P)|(T_{\omega}^{-1}+ST^{-1}+T_{\omega}ST_{\omega}^{-1}+ST_{\omega}^{-1}ST^{-1})\\
		=&(P|S)|(T_{\omega}^{-1}+ST^{-1}+T_{\omega}ST_{\omega}^{-1}+ST_{\omega}^{-1}ST^{-1})
	\end{align*}
	which shows that
	\begin{gather*}
		P|(T^{-1}+ST_\omega^{-1} + T_{\omega}^{-1}ST^{-1}+ST_{\omega}ST_{\omega}^{-1})=0.
	\end{gather*} It follows that $P|\e \in ker(T+ST_\omega + T_{\omega}ST+ST_{\omega}^{-1}ST_{\omega}).$
\end{proof}

\begin{prop}
	The polynomial $P_{k,\D}$ belongs to the subspace $W_{k,k}^1$ and therefore the map $f:\Gamma_{7} \to V_{k,k}$ defined by 
	\begin{gather*}
		\gamma \mapsto (\H|(\1-S))|\gamma
	\end{gather*} is a cocycle belonging to $C_p(\Gamma_{7},V_{k,k})$.
\end{prop}
\begin{proof}
	The proof that $P_{k,\D} \in ker(\1+S)$ and $P_{k,\D} \in ker(\1+U+U^2)$ is same as the one given in the proof of Proposition \ref{W1+}. Let us now calculate $P_{k,\D}|(T+ST_\omega + T_{\omega}ST+ST_{\omega}^{-1}ST_{\omega})$ which is equal to
	\begin{align*}
		=&(\H|(\1-S))|(T+ST_\omega + T_{\omega}ST+ST_{\omega}^{-1}ST_{\omega})\\
		=&\H|(T+ST_\omega + T_{\omega}ST+ST_{\omega}^{-1}ST_{\omega}-ST-T_\omega -ST_{\omega}ST-T_{\omega}^{-1}ST_{\omega})\\
		=&\H|(T+ST_\omega + T_{\omega}ST+TT_{\omega}^{-1}ST_{\omega}ST-ST-T_\omega -ST_{\omega}ST-T_{\omega}^{-1}ST_{\omega}).
	\end{align*}
	Since $\H|T_{\omega}=\H$ and $\H|T_{\omega}^{-1}=\H$ we see that
	\begin{gather*}
		\H|(T+ST_\omega + T_{\omega}ST+TT_{\omega}^{-1}ST_{\omega}ST-ST-T_\omega -ST_{\omega}ST-T_{\omega}^{-1}ST_{\omega})=0
	\end{gather*}
	which shows that $P_{k,\D} \in ker(T+ST_\omega + T_{\omega}ST+ST_{\omega}^{-1}ST_{\omega})$. Finally, by part (2) of Proposition \ref{PPolyforZ[i]prop} we have
	\begin{gather*}
		P_{k,\D}|\e=P_{k,\D}
	\end{gather*}
	which concludes the proof that $P_{k,\D}\in W_{k,k}^{1} $.
\end{proof} 

Unlike the previous cases, we calculated the dimensions of $ W_{k,k}^{1}$ for $1\leq k \leq 27$ due to lack of computing power at our disposal. In fact, it took over a week on my laptop for the computer program that I wrote to compute the dimension of  $W_{k,k}^{1}$ for $k=27$. Regardless, here is the table containing our data for the case of $d=7$.

\begin{table}[h]
	\centering
	\begin{tabular}{c rrrrrrrrrrrrrrrrr}
		
		$k$ & \vline & 1 & 3 & 5 & 7 & 9 & 11 & 13 & 15 & 17 & 19 & 21 & 23 & 25 & 27 \\ 
		\hline
		$\textbf{dim}(W_{k,k}^{1})$ & \vline & 1 & 1 & 2 & 3 & 3  & 4 & 5 & 5 & 6 & 7 & 7 & 8 & 9 & 9 \\
	\end{tabular}
\end{table}
We also observed the same behavior as the previous two cases with respect to the dimensions of $\textbf{dim}(W_{k,k})$ and $\textbf{dim}(W_{k,k}^{-1})$ which is $\textbf{dim}(W_{k,k})=\textbf{dim}(W_{k,k}^{1})$ and $\textbf{dim}(W_{k,k}^{-1})=0 $ for $1\leq k \leq 19$. We again record the following conjecture based on the computational evidence.
\begin{conjec}
	For odd $k\geq 1$ we have that $\displaystyle\textbf{dim}(W_{k,k}^{1})= \left\lfloor \frac{k-1}{3} \right\rfloor+1 .$	
\end{conjec}

\hskip 1in

\subsection{ The case $d=11$ }
We use the following presentation of $\Gamma_{11}$ given in \cite{Fine}
\begin{gather*}
	\Gamma_{11}=\langle\, S,T,T_{\omega} \,\big|\,  S^2=(ST)^3=(T_{\omega}^{-1}ST_{\omega}ST)^3=[T,T_\omega]=\1\rangle
\end{gather*}
As usual, the space of parabolic cocycles on $\Gamma_{11}$ is defined to be 
\begin{gather*}
	C_{p}(\Gamma_{11},V_{k,k})=\{f \in C(\Gamma_{11},V_{k,k})\, |\, f(T)=f(T_{\omega})=0 \},
\end{gather*}
which again can be identified it with a subspace of $V_{k,k}$ as follows. 
\begin{prop}\label{W11} The map sending $f\in C_{p}(\Gamma_{11},V_{k,k}) $ to $P_S \in V_{k,k}$ is a $\CC$-vector space isomorphism from $C_{p}(\Gamma_{11},V_{k,k})$ to $W_{k,k}$ where	
	\begin{gather*}
		W_{k,k}:=ker(\1+S)\cap ker(\1+U+U^2)\cap ker(T+ST_\omega +TE+ST_\omega E^{-1}+ T_{\omega}ST+ST_{\omega}^{-1}ST_{\omega}),
	\end{gather*} with $U=TS$ and $E=T_{\omega}^{-1}ST_{\omega}ST$.
\end{prop}
\begin{proof}
	Let $f\in C_{p}(\Gamma_{11},V_{k,k})$ such that $f(\gamma)=P_{\gamma}\in V_{k,k}$ and assume that $f\mapsto f(S)=P_S$. Again, the fact that
	\begin{gather*}
		P_S\in ker(\1+S), \hskip 0.1in P_S\in ker(\1+U+U^2)
	\end{gather*}
	follows as before.  For the proof of $P_S\in ker(T+ST_\omega +TE+ST_\omega E^{-1}+ T_{\omega}ST+ST_{\omega}^{-1}ST_{\omega})$, we observe that
	\begin{gather*}
		f(T_{\omega}^{-1}ST_{\omega}ST)=f(E)=P_S|T_{\omega}ST+P_S|T.
	\end{gather*}
	The relation $(T_{\omega}^{-1}ST_{\omega}ST)^3=E^3=\1$ gives
	\begin{align*}
		0&=f((T_{\omega}^{-1}ST_{\omega}ST)^3)\\
		&=f(E)|E^2+f(E)|E+F(E)\\
		&=(P_S|T_{\omega}ST+P_S|T)|E^{-1}+(P_S|T_{\omega}ST+P_S|T)|E+P_S|T_{\omega}ST+P_S|T \quad (E^2=E^{-1}),
	\end{align*}
	and since 
	\begin{gather*}
		T_{\omega}STE^{-1}=ST_{\omega}, \quad TE^{-1}=ST_{\omega}^{-1}ST_{\omega} 
	\end{gather*}
	the result follows.	Finally, we do not get any new equations from the remaining relations and the rest of the proof proceeds as in the proof of Proposition \ref{W1}.
\end{proof}

$W_{k,k}$ is stable under the action of $\e=\begin{psmallmatrix}
-1&0\\ 0&1 
\end{psmallmatrix}$ and splits as in the case of $d=2,7$. 
\begin{prop}\label{W11-decomp}
	$W_{k,k}|\e =W_{k,k}$ and hence 
	\begin{gather*}
		W_{k,k}=W_{k,k}^{1}\oplus W_{k,k}^{-1}
	\end{gather*} where $W_{k,k}^{\pm 1}=W_{k,k}\cap V_{k,k}^{\pm 1} $. 
\end{prop} 

\begin{proof}
	Let $P\in W_{k,k}$. We just need to show that $P|\e$ belongs to $ker(T+ST_\omega +TE+ST_\omega E^{-1}+ T_{\omega}ST+ST_{\omega}^{-1}ST_{\omega})$ since the proof of the fact that $P|\e$ belongs to $ker(\1+S)$ and  $ker(\1+U+U^2)$ is same as before because the identities 
	\begin{gather*}
		\e S= S\e, \quad \e U=SU^2S\e
	\end{gather*}
	still hold. 
	
	So, using
	\begin{gather*}
		\e T=T^{-1}\e,\quad \e T_\omega=T_\omega^{-1}\e,
	\end{gather*}	
	along with $E=T_{\omega}^{-1}ST_{\omega}ST$ and $E^{-1}=T^{-1}ST_{\omega}^{-1}ST_{\omega}$, to calculate $(P|\e)|(T+ST_\omega +TE+ST_\omega E^{-1}+ T_{\omega}ST+ST_{\omega}^{-1}ST_{\omega})$ gives  
	\begin{align*}
		&=(P|\e)|(T+ST_\omega +TT_{\omega}^{-1}ST_{\omega}ST+ST_\omega T^{-1}ST_{\omega}^{-1}ST_{\omega}+ T_{\omega}ST+ST_{\omega}^{-1}ST_{\omega})\\
		&=(P|(T^{-1}+ST_{\omega}^{-1}+T^{-1}T_{\omega}ST_{\omega}^{-1}ST^{-1}+ST_{\omega}^{-1}TST_{\omega}ST_{\omega}^{-1}+T_{\omega}^{-1}ST^{-1}+ST_{\omega}ST_{\omega}^{-1}))|\e
	\end{align*}
	Now, since 
	\begin{gather*}
		S^2=\1,\quad T_\omega^{-1}T^{-1}=T^{-1}T_\omega^{-1},\quad \text{and}\quad P|S=-P,
	\end{gather*}
	we have
	\begin{align*}
		0=&(-P)|(T+ST_\omega +TT_{\omega}^{-1}ST_{\omega}ST+ST_\omega T^{-1}ST_{\omega}^{-1}ST_{\omega}+ T_{\omega}ST+ST_{\omega}^{-1}ST_{\omega})\\
		=&(-P)|(T+ST_\omega +TT_{\omega}^{-1}ST_{\omega}ST+ST_\omega T^{-1}ST_{\omega}^{-1}ST_{\omega}+ T_{\omega}ST+ST_{\omega}^{-1}ST_{\omega})|T_{\omega}^{-1}T^{-1}\\
		=&(-P)|(T_{\omega}^{-1}+ST^{-1}+TT_{\omega}^{-1}ST_{\omega}ST_{\omega}^{-1}+ST_\omega T^{-1}ST_{\omega}^{-1}ST^{-1}+T_{\omega}ST_{\omega}^{-1}+ST_{\omega}^{-1}ST^{-1})\\
		=&(P|S)|(T_{\omega}^{-1}+ST^{-1}+TT_{\omega}^{-1}ST_{\omega}ST_{\omega}^{-1}+ST_\omega T^{-1}ST_{\omega}^{-1}ST^{-1}+T_{\omega}ST_{\omega}^{-1}+ST_{\omega}^{-1}ST^{-1})\\
		=&P|(ST_{\omega}^{-1}+T^{-1}+STT_{\omega}^{-1}ST_{\omega}ST_{\omega}^{-1}+T_\omega T^{-1}ST_{\omega}^{-1}ST^{-1}+ST_{\omega}ST_{\omega}^{-1}+T_{\omega}^{-1}ST^{-1}).
	\end{align*}
	But,
	\begin{align*}
		STT_{\omega}^{-1}ST_{\omega}ST_{\omega}^{-1}&=ST_{\omega}^{-1}TST_{\omega}ST_{\omega}^{-1}\\
		T_\omega T^{-1}ST_{\omega}^{-1}ST^{-1}&=T^{-1}T_{\omega}ST_{\omega}^{-1}ST^{-1}
	\end{align*}
	which shows that
	\begin{gather*}
		P|(T^{-1}+ST_{\omega}^{-1}+T^{-1}T_{\omega}ST_{\omega}^{-1}ST^{-1}+ST_{\omega}^{-1}TST_{\omega}ST_{\omega}^{-1}+T_{\omega}^{-1}ST^{-1}+ST_{\omega}ST_{\omega}^{-1})=0.
	\end{gather*} It follows that $P|\e \in ker(T+ST_\omega +TE+ST_\omega E^{-1}+ T_{\omega}ST+ST_{\omega}^{-1}ST_{\omega}).$
\end{proof}

\begin{prop}
	The polynomial $P_{k,\D}$ belongs to the subspace $W_{k,k}^1$ and therefore the map $f:\Gamma_{11} \to V_{k,k}$ defined by 
	\begin{gather*}
		\gamma \mapsto (\H|(\1-S))|\gamma
	\end{gather*} is a cocycle belonging to $C_p(\Gamma_{11},V_{k,k})$.
\end{prop}
\begin{proof}
	The proof that $P_{k,\D} \in ker(\1+S)$ and $P_{k,\D} \in ker(\1+U+U^2)$ is same as before. To show that $P_{k,\D} \in ker(T+ST_\omega +TE+ST_\omega E^{-1}+ T_{\omega}ST+ST_{\omega}^{-1}ST_{\omega})$ we calculate
	\begin{gather*}
		(\H|(\1-S))|(T+ST_\omega +TE+ST_\omega E^{-1}+ T_{\omega}ST+ST_{\omega}^{-1}ST_{\omega}).
	\end{gather*}
	Again, using the identities $\H|T_{\omega}=\H$ and $\H|T_{\omega}^{-1}=\H$ together with $E=T_{\omega}^{-1}ST_{\omega}ST$ and $E^{-1}=T^{-1}ST_{\omega}^{-1}ST_{\omega}$ we end up with 
	\begin{gather*}
		(\H|(\1-S))|(T+ST_\omega +TE+ST_\omega E^{-1}+ T_{\omega}ST+ST_{\omega}^{-1}ST_{\omega})=\H|(ST_\omega E^{-1}-STE).
	\end{gather*}	
	But
	\begin{gather*}
		ST_\omega E^{-1}=T_{\omega}STE,
	\end{gather*}
	and so 	 
	\begin{gather*}
		\H|(ST_\omega E^{-1}-STE)=\H|(T_{\omega}STE-STE)=0
	\end{gather*}		
	which shows that $P_{k,\D} \in ker(T+ST_\omega +TE+ST_\omega E^{-1}+ T_{\omega}ST+ST_{\omega}^{-1}ST_{\omega})$. Finally, by part (2) of Proposition \ref{PPolyforZ[i]prop} we have
	\begin{gather*}
		P_{k,\D}|\e=P_{k,\D}
	\end{gather*}
	which concludes the proof that $P_{k,\D}\in W_{k,k}^{1} $.
\end{proof} 
Again, the lack of computational resources and the complicated description of the third kernel limited our dimension calculations to the range $1\leq k \leq 21$. In this range the data again showed that  $\textbf{dim}(W_{k,k})=\textbf{dim}(W_{k,k}^{1})$ and $\textbf{dim}(W_{k,k}^{-1})=0$. Here is the table containing our data for $d=11$ case.
\begin{table}[h]
	\centering
	\begin{tabular}{c rrrrrrrrrrrrrrrrr}
		$k$ & \vline & 1 & 3 & 5 & 7 & 9 & 11 & 13 & 15 & 17 & 19 & 21\\ 
		\hline
		$\textbf{dim}(W_{k,k}^{1})$ & \vline & 1 & 2 & 3 & 4 & 5  & 6 & 7 & 8 & 9 & 10 & 11\\
	\end{tabular}
\end{table}\\
Based on this numerical data we are led to the following conjecture.
\begin{conjec}
	For odd $k\geq 1$ we have that $\displaystyle\textbf{dim}(W_{k,k}^{1})= \frac{k+1}{2}.$	
\end{conjec}

\hskip 1in

\subsection{Some concluding remarks and future work}
We close this section by elaborating on the importance of the conjectural formulas we have stated and how they relate to Bianchi cusp forms which are the analogs of cusp forms for imaginary quadratic fields. We also list some questions (some of which are being considered in \cite{FKW}) that may have occurred to the reader but are left unanswered. 

We will not define most of the objects we refer to in this subsection and be intentionally vague about them due to the fact giving precise definitions require a lot of sophisticated mathematical machinery. We also do not want to stray too much from the main theme of this paper which is to study the function $\H$. However, some good references that discuss some of the things we refer to are \cite{Ko-Za,Mo,Sen1}. 

We begin with the generalized Eichler-Shimura Isomorphism which states that  
\begin{gather*}
	H^{1}_{\text{cusp}}(\Gamma,V_{n,n}) \simeq S_{n+2}(\Gamma)
\end{gather*}
where $\Gamma$ is a Bianchi group, $S_n(\Gamma)$ is the vector space of Bianchi cusp forms of weight $n$ on $\Gamma$, and $H^{1}_{\text{cusp}}$ refers to the cuspidal cohomology of $\Gamma$. By translating the sheaf cohomology into group cohomology one can show that
\begin{gather*}
	H^{1}_{\text{cusp}}(\Gamma,V_{n,n}) \simeq H^{1}_{\text{par}}(\Gamma,V_{n,n})
\end{gather*}
where $H^{1}_{\text{par}}$ is the parabolic cohomology defined as the quotient of parabolic cocycles on $\Gamma$ by parabolic coboundaries on $\Gamma$. We already saw that when $\Gamma=\Gamma_d$ the space of parabolic cocycles can be identified with $W_{k,k}$. Also, one can easily show that the space of parabolic coboundaries can be identified with constant polynomials implying that it is of dimension 1. Consequently, 
\begin{gather*}
	\textbf{dim}(S_{k}(\Gamma_d))=\textbf{dim}(W_{k-2,k-2})-1.
\end{gather*} 
So, if the conjectural formulas we have stated are true then, at least in the case when $k$ is odd, one immediately obtains explicit formulas, depending only on $k$, for the dimension of $S_{k}(\Gamma_d)$. As far as we can tell, no explicit dimension formulas for $S_{k}(\Gamma)$ are known. However, Finis et al.\cite{FGT} have carried out extensive computations on the dimension of various cohomology groups associated to Bianchi groups which provides a way of checking the validity of our data. We are encouraged by the fact that our data does agree with theirs in the overlapping cases. Indeed, we checked our data for the  
$\textbf{dim}H^{1}_{\text{cusp}}(\Gamma_d,V_{k,k})$ against the data\footnote{The table gives the dimension of $\textbf{dim}H^{1}(\Gamma_d,V_{k,k})$, but the $\textbf{dim}H^{1}_{\text{cusp}}(\Gamma_d,V_{k,k})$ can be easily obtained using the formula $$\textbf{dim}H^{1}(\Gamma_d,V_{k,k})-\textbf{dim}H^{1}_{\text{cusp}}(\Gamma_d,V_{k,k})=\nu_{K,k}h_K $$ which is given on the same page.} provided in Table 1 on pg. 53 in \cite{FGT}, and both sets of data are in agreement for $k$ odd and in the range $1\leq k\leq15$.

Finally, we discuss some future work and questions:
\begin{itemize}
	\item[1)] The foremost question is obviously whether the formulas we have stated are true or not. Currently, we are working to answer this question in \cite{FKW}. We are also trying to find and prove dimension formulas when $k$ is even by looking at the experimental data. Another aspect of our work in \cite{FKW} is to define the analog of the \textit{period map} from the classical case of $\textbf{PSL}(2,\mathbb{Z})$ for the Bianchi cusp forms.	
	\item[2)] There is an action of Hecke operators on $H^{1}_{\text{cusp}}(\Gamma,V_{n,n})$ and the isomorphism 
	\begin{gather*}
		H^{1}_{\text{cusp}}(\Gamma,V_{n,n}) \simeq S_{n+2}(\Gamma)
	\end{gather*} is a Hecke module isomorphism. It will be interesting to transfer the action of Hecke operators to the space $W_{k,k}$ and explore its consequences. Indeed, Zagier \cite{Za3} has done this for the space of period polynomials which is defined similarly as $W_{k,k}$ and has obtained interesting results for the traces of Hecke operators. We plan to carry this out in \cite{FKW} or in a subsequent paper.
	
	\item[3)] Reader may be wondering if there are any other values of $k$, other than the ones we have found, for which $\H$ is constant. The anonymous referee also asked about this. The answer is most likely no since it is shown in \cite{FGT} that the dimension of $H^{1}_{\text{cusp}}(\Gamma,V_{n,n})$ grows at least linearly in $n$ which means that the dimension of the space of parabolic cocycles containing $P_{k,\D}$ grows at least linearly in $n$ as well, and $\H$ is constant precisely when the subspace containing $P_{k,\D}$ in the direct sum decomposition of the space of parabolic cocycles is one dimensional. This, along with computational evidence we have gathered, strongly suggest that there are no other $k$ values, other than the ones we have found, for which $\H$ is constant.
	
	\item[4)] Finally, the reader may also be wondering whether the kernels of the linear maps defining $W_{k,k}$ can be combined into a single or fewer kernels. We have not pursued this, but it is likely to be true. Indeed, this is true for the space of parabolic cocycles on $\textbf{PSL}(2,\mathbb{Z})$ defined as the intersection of $ker(\1+S)$ and $ker(\1+U+U^2)$ where $S$ and $U$ are the generators of $\textbf{PSL}(2,\mathbb{Z})$. It can be easily checked, which is left as an  exercise in \cite{Za2}, that
	\begin{gather*}
		f\in ker(\1+S) \cap ker(\1+U+U^2) \iff f \in ker(\1-US-U^2S)
	\end{gather*}
	where $f$ is a parabolic cocycle on $\textbf{PSL}(2,\mathbb{Z})$.
\end{itemize}

\end{document}